\documentclass[3p,11pt,authoryear]{elsarticle}
\usepackage{lipsum}
\makeatletter
\def\ps@pprintTitle{%
 \let\@oddhead\@empty
 \let\@evenhead\@empty
 \def\@oddfoot{}%
 \let\@evenfoot\@oddfoot}
\makeatother

\usepackage{fancyhdr}
\usepackage{amssymb}
\usepackage{multirow}
\usepackage{mathtools} %[fleqn]
\let\oldforall\forall
\let\forall\undefined
\DeclareMathOperator{\forall}{\oldforall}
\usepackage{amsmath}
\usepackage{blindtext}
\usepackage{nccmath}
\usepackage{amsthm}
\usepackage{multicol}
\usepackage{subcaption}
\usepackage{float}
\usepackage{lscape}
\usepackage{xcolor,soul}
\usepackage[normalem]{ulem}
\usepackage{changepage}
\usepackage{algorithm}
\usepackage{comment}
\usepackage{float}
\usepackage{setspace}
\usepackage[figuresright]{rotating}
\usepackage{graphicx}
\usepackage{lineno}
\usepackage{natbib}
\usepackage{makecell}
\usepackage{bbm}
\usepackage{accents}
\usepackage{mathptmx}
\usepackage{geometry}
\usepackage{bm}
\usepackage{mathrsfs}

\DeclareUnicodeCharacter{00A0}{~}

\usepackage{algorithm}
\usepackage[noend]{algpseudocode}
\usepackage[english]{babel}
\newtheorem{theorem}{Theorem}
\newtheorem{prop}{Proposition}

\newtheorem{lemma}[theorem]{Lemma}

\theoremstyle{remark}

\makeatletter
\def\BState{\State\hskip-\ALG@thistlm}
\makeatother
 
\usepackage{amsthm}
\usepackage{amssymb}
\usepackage[hidelinks]{hyperref}

\begin{document}

\begin{frontmatter}



%\title{Optimal Blocking and Station Design for Long Urban Railway Trains}

%\title{Joint Optimization of Train Block Design and Timetable Plan for Extra-Long Metro Trains}

% another possible title, possibly more direct
\title{Dynamic Random Bipartite Matching under Spatiotemporal Heterogeneity: General Models and Application to Mobility Services}


\author[label1]{Shiyu Shen}
\author[label1]{Yanfeng Ouyang}

\address[label1]{Department of Civil and Environmental Engineering, University of Illinois at Urbana-Champaign, Urbana, IL 61801, USA}

\begin{abstract}
This paper explores a variant of bipartite matching problem, referred to as the Spatiotemporal Random Bipartite Matching Problem (ST-RBMP), that accommodates randomness and heterogeneity in the spatial distributions and temporal arrivals of bipartite vertices. This type of problem can be applied to many location-based services, such as shared mobility systems, where randomly arriving customers and vehicles must be matched dynamically. 
This paper proposes a new modeling framework to address ST-RBMP's challenges associated with the spatiotemporal heterogeneity, dynamics, and stochastic decision-making. 
The objective is to dynamically determine the optimal vehicle/customer pooling intervals and maximum matching radii that minimize the system-wide matching costs, including customer and vehicle waiting times and matching distances. 
Closed-form formulas for estimating the expected matching distances under a maximum matching radius are developed for static and homogeneous RBMPs, and then extended to accommodate spatial heterogeneity via continuum approximation. The ST-RBMP is then formulated as an optimal control problem where optimal values of pooling intervals and matching radii are solved over time and space. 
A series of experiments with simulated data are conducted to demonstrate that the proposed formulas for static RBMPs under matching radius and spatial heterogeneity yield very accurate results on estimating matching probabilities and distances. Additional numerical results are presented to demonstrate the effectiveness of the proposed ST-RBMP modeling framework in designing dynamic matching strategies for mobility services under various demand and supply patterns, which offers key managerial insights for mobility service operators.
\end{abstract}




\end{frontmatter}

%\linenumbers
% main text
\section{%Section Title Capitalizes All ``Major'' Words
Introduction}

The bipartite matching problem is fundamental in the field of applied mathematics and combinatorial optimization. The classic problem considers a bipartite graph with two disjoint subsets of vertices, and the objective is to find an optimal subset of edges that match the vertices into disjoint pairs. In the past decade, the online bipartite matching problem, a dynamic variation of the classic problem, has received significant attention.
This is driven by advances in enabling information and communication technologies \citep{mehta_online_2013}, as well as a wide variety of application contexts, such as interactions between users/information in social media \citep{wu_graph_2022}, e-commerce \citep{zhou_bipartite_2007}, and crowd-sourcing services \citep{zha_economic_2016}. 
Unlike the classic problem where all vertices and edges are static, in the online problem, one or both subset(s) of the vertices arrive dynamically. % in the online problem. 
Upon arrival of each vertex, a decision will be made on whether to match it with an available vertex from the other subset or leave it unmatched for future opportunities. 
Many strategies and algorithms have been developed to solve these problems, including approximation algorithms \citep{feng_batching_2020, shanks_online_2022, shanks_approximation_2023}, dynamic programming approaches \citep{psaraftis_dynamic_2016}, and meta-heuristics \citep{najmi_novel_2017}. 

This paper explores a variation of the online bipartite matching problem that further incorporates randomness in the spatiotemporal distributions of the vertices, which is referred to as the Spatiotemporal Random Bipartite Matching Problem (ST-RBMP).
This problem features two distinct assumptions: (i) the vertices in the bipartite graph are randomly distributed in space, and the edge weights between vertices are measured by a spatial metric; 
(ii) the vertices in both subsets are revealed dynamically over time according to certain processes.
Without loss of generality, we refer to the vertices in the smaller subset as ``demand" vertices, those in the larger subset as ``supply" vertices, and break ties arbitrarily.
This type of problem directly builds upon the spatiotemporal information of the vertices, and can be applied in many contexts; e.g., matching customers with vehicles for shared mobility services \citep{shen_dynamic_2023}, assigning patients to healthcare providers \citep{rao_surf_2020}, %distributing tasks to workers \citep{cheng_prediction-based_2017},
and distributing customers or tasks to a set of servers \citep{afeche_ride-hailing_2018}. 

% The challenges
The randomness of vertex distributions blurs the structure of effective matching strategies, particularly when the distributions are heterogeneous. The associated challenges are twofold. First, the spatial heterogeneity indicates that different neighborhoods have varying levels of demand and supply, as commonly observed in real-world mobility systems \citep{yang_modeling_2017}, which raises questions about how to 
%tailor the 
balance between matches that are within vs. across different neighborhoods. 
%prioritize matching of vertices across different neighborhoods and whether tailored strategies are needed for specific neighborhoods. 
For example, the optimal matching strategies in densely populated city centers may differ from those in sparsely populated suburban areas. 
Second, the temporal dynamics of supply and demand arrivals/departures force that matching decisions be continuously adapted to the evolving system states.
Improper decision-making in such dynamic and stochastic systems could lead to undesirable consequences. 
For instance, shared mobility systems often suffer from the so-called wild goose chase (WGC) phenomenon \citep{arnott_taxi_1996, daganzo_public_2010,castillo_surge_2017}, where a large number of vehicles are trapped in long unproductive deadheading from their locations to customer origins. 
This inefficient situation significantly compromises resource utilization and system performance.  If not properly managed, the system can remain in such an unfavorable state for a significant amount of time \citep{ouyang_measurement_2023}.
Many believe that low-quality vehicle-customer matching (e.g., due to instantaneous one-to-many matching) is the main cause of WGC, and to enhance system performance, a variety of dynamic routing and dispatching strategies have been proposed, including path-based vehicle rerouting \citep{lei_path-based_2019, shen_path-based_2021}, empty vehicle repositioning \citep{ke_system_2021}, dynamic vehicle swapping \citep{ouyang_measurement_2023, shen_dynamic_2023}, and optimization-based re-assignment \citep{maciejewski_assignment-based_2016, alonso-mora_-demand_2017, hyland_dynamic_2018}.

In addition to these tactical-level strategies and algorithms, it is critical to develop a systematic approach to determine the hyper-parameters of ST-RBMP that control the implementation of the matching process. %, such as the optimal timing for matches, and candidate demand or supply vertices that may be considered for matching. 
%Many t\
Transportation researchers have proposed the possibility of imposing (i) a supply-demand pooling interval to control optimal timing for matches, and (ii) a maximum matching radius to screen candidate demand or supply vertices that may be considered for matching (e.g., see 
\citet{yang_optimizing_2020}).
In Figure \ref{fig:radius and pooling},  demand and supply vertices (represented by the square and cross markers) arrive dynamically, and matching decisions are made at a sequence of decision epochs. 
The cumulated demand and supply vertices, including those new arrivals after the most recent matching epoch and those ``leftovers" from all previous matchings, form a new matching problem instance. A longer pooling interval between matching epochs could potentially include more vertices and reduce the average matching distance, but it also increases the expected waiting-for-match time for all these vertices.
At each decision epoch, the maximum matching radius dictates that a demand vertex can be considered for matching with supply vertices within that radius. 
A larger maximum radius can increase the number of successful matches at a single epoch but may result in longer average matching distances, and vice versa. 
\begin{figure}[ht] 
    \centering        
    \includegraphics[width=0.80\textwidth]{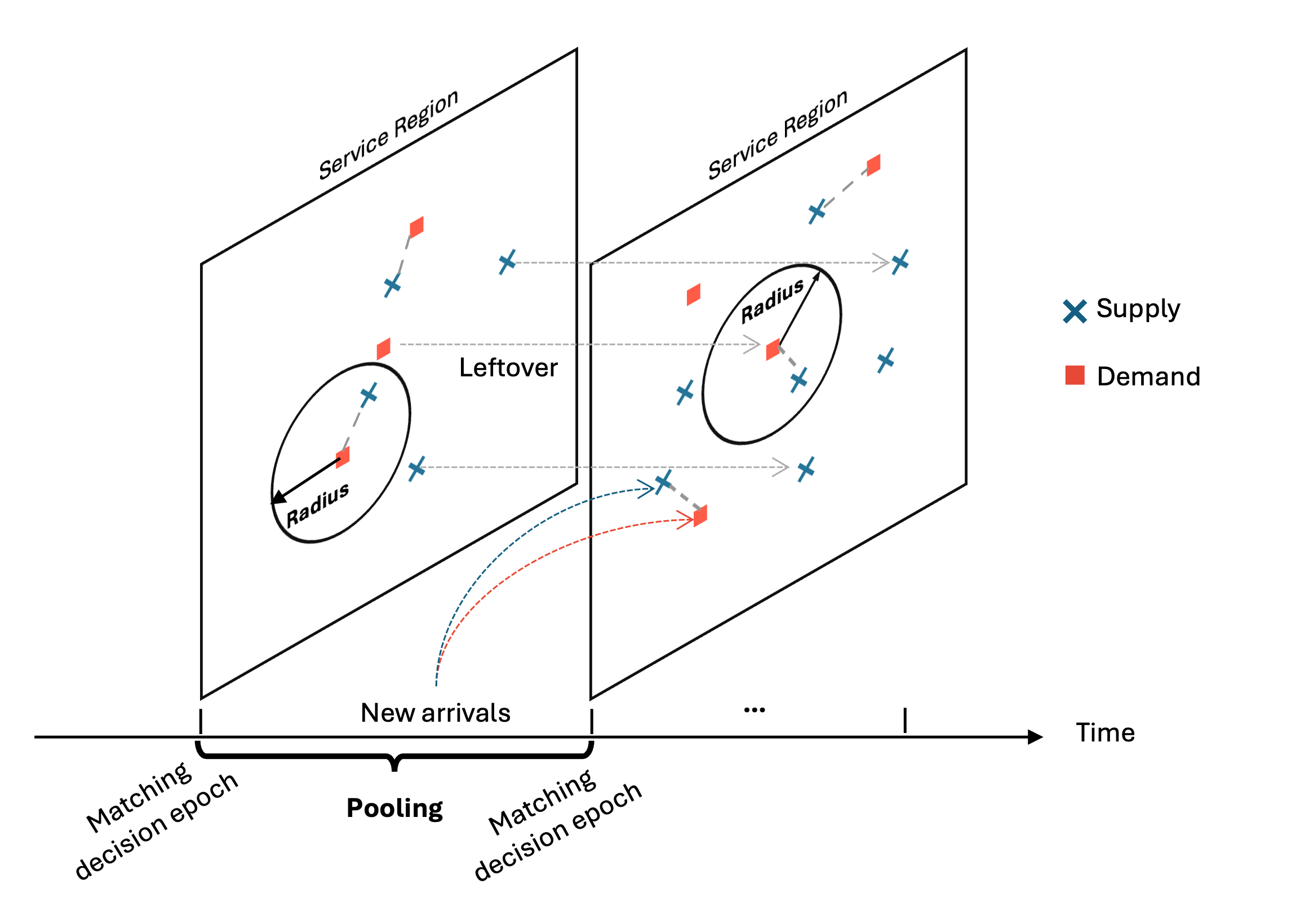}
    \caption{Matching Radius and Pooling Interval.}
    \label{fig:radius and pooling}
\end{figure}

Efforts have been made to optimize pooling intervals and maximum matching radii under specific problem settings. 
When considering a single decision epoch or assuming the system is at an equilibrium state, the problem reduces to a static RBMP. 
In a static RBMP with homogeneous vertex distributions, analytical models have been developed to analyze the impacts or determine the optimal values of either the matching radius \citep{xu_supply_2020}, the pooling interval \citep{shen_zhai_ouyang_2024}, or both \citep{yang_optimizing_2020}, based on estimates of the expected matching distance between matched vertices. However, no studies have extended the analysis to develop analytical models for problems with heterogeneous vertex distributions; %, where the is significantly more challenging. 
%For ST-RBMP under a heterogeneous setting
rather, some studies have adopted data-driven approaches to learn the optimal matching policies from historical data, such as \citet{qin_optimizing_2021} and \citet{liang_enhancing_2023}. 
While these learning methods may capture spatial and temporal heterogeneity from real-world systems, they require extensive data input, impose heavy computational burden, and face challenges related to transferability and robustness across problem settings. In many cases, insights from analytical models are preferable as they provide more concise and interpretable results, as well as ease to draw managerial insights.

To address all these challenges, this paper first proposes new closed-form formulas, as building blocks, for estimating the expected matching probability and distance in static RBMPs under maximum matching radii and/or spatially heterogeneous vertex distributions. This is achieved by revealing a desirable scaling property of homogeneous RBMPs: when the numbers of demand and supply vertices are not (nearly) balanced, the expected matching distance becomes largely independent of the size of the matching region but rather depends primarily on local vertex densities. A series of Monte Carlo simulations are conducted to verify this scaling property and demonstrate that the proposed formulas provide highly accurate estimates across a wide range of problem settings.

Building on the analytical formulas for static RBMPs, we next propose a time- and space-dependent control framework which dynamically determines the optimal pooling intervals and matching radii for ST-RBMP. The optimal control problem is formulated within a continuum approximation scheme, and the optimal values of the control variables over time and space are solved from local optimality conditions. %established by Pontryagin's Minimum Principle (PMP).
A set of numerical experiments is conducted to demonstrate the effectiveness of the proposed ST-RBMP modeling framework in designing dynamic matching strategies for mobility services. 
The results show that the framework not only effectively captures and addresses spatiotemporal heterogeneity in demand and supply distributions, but also provides a theoretical explanation on the impacts of matching radius and pooling interval under various service scenarios. 
%The effectiveness of the proposed modeling framework is verifyd through a series of numerical experiments.

Finally, both analytical and numerical results in this paper provide valuable managerial insights for mobility service operators. 
For example, in a closed-loop system with a fixed fleet size (with balanced customer and vehicle %demand and supply 
arrival rates), it may already be optimal to use instantaneous matching (i.e., no pooling) without imposing a matching radius. In contrast, for an open-loop system, when vehicles are expected to arrive at a higher rate than customers in the future, pooling customers and vehicles (through delayed matching) and imposing a dynamically adjusted spatial-dependent matching radii (i.e., be ``picky" on matches) can sometimes be beneficial.%, and the matching radii should be dynamically adjusted as the system evolves.

The remainder of this paper is organized as follows. Section \ref{sec:RBMP_static} focuses on static and homogeneous RBMPs and presents the scaling property, and the matching probability and distance formulas under maximum matching radius.
Section \ref{sec:RBMP_spatial_heter} then extends these matching probability and distance formulas to static and heterogeneous RBMPs. Section \ref{sec:RBMP_dynamic} formulates the ST-RBMP as a dynamic optimal control problem using the analytical results from Section \ref{sec:RBMP_spatial_heter}; a solution approach is also proposed.
Section \ref{sec:numerical} presents numerical experiments to verify the effectiveness of the proposed formulas and modeling framework. Finally, Section \ref{sec:conclusion} concludes the paper and outlines several directions for future research.

%\subsection{%Subsection capitalizes only the first word
%Expected Bipartite Matching Distance}

%\begin{table}[h]
%    \centering
%    \caption{An example of a table.}
%    \label{tab1}  % Place \label AFTER the caption
%    \begin{tabular}{l l l}
%        \hline
%        Column heading & Column A & Column B \\
%        \hline
%        Entry & 1 & 2 \\
%        Another entry & 3 & 4 \\
%        \hline
%    \end{tabular}
%\end{table}

\section{Static Homogeneous RBMPs}% under Spatial Heterogeneity}
\label{sec:RBMP_static}
In this section, we employ the recent formulas for static and homogeneous RBMPs \citep{shen_zhai_ouyang_2024} to account for two problem extensions. We begin by showing a scaling property of the expected matching distance with respect to the size of the spatial region. 
Then, we develop new formulas for estimating the matching probability and expected matching distance under a maximum allowable matching radius. 
These analytical results lays a theoretical foundation for deriving approximate formulas for both matching probability and expected matching distance of a heterogeneous RBMP based on local vertex densities and a locally imposed maximum matching radius. 

\subsection{Existing results}
A static and homogeneous RBMP is formally defined as follows. Consider a realization of $m$ demand vertices and $n$ supply vertices randomly and uniformly distributed within a given region, where the cost (or weight) of a match is measured by the distance between the vertices. Without loss of generality, we assume $n \ge m$. In each realized instance of the problem, every demand vertex is matched to exactly one supply vertex. The problem seeks an optimal set of matches that minimizes the total distance across all matched pair of vertices. % will be identified. 
The RBMP is defined over all possible realizations of such instances, and the objective is to estimate the distribution and moments of the optimal matching distance per demand vertex, denoted by a random variable $X$. 

\citet{shen_zhai_ouyang_2024} developed analytical models that yield closed-form approximate formulas for RBMPs %to address this problem 
within a ``unit-volume" hyper-ball in a $D$-dimensional L$^p$ space. The key steps involve:
\begin{itemize}
    \item[(i)] deriving the probability that a randomly selected demand vertex is matched to its $k$-th nearest supply vertex, denoted by $\mathbb{P}(k)$; and 
    \item[(ii)] deriving the distribution of the conditional expected distance from a randomly selected demand vertex to its $k$-th nearest neighbor, denoted by $X_k$.
\end{itemize}
According to \citet{shen_zhai_ouyang_2024}, %when $D=2$ or $D = 3$, 
the cumulative distribution function (CDF) and the $M$-th moment of the optimal matching distance $X$, denoted by $F_X(x)$ and $\mathbb{E}[X^M]$, respectively, can be approximately estimated by the following formulas:
\begin{align}
    F_X(x) &\approx \sum_{k=1}^m \mathbb{P}(k)\cdot F_{X_k}(x)
    = \sum_{k=1}^m \mathbb{P}(k)\cdot I_{(\frac{x}{R})^D}(k,n-k+1), 
    \label{eq:F_X_general}\\
    \mathbb{E}[X^M] 
    &\approx 
    \sum_{k=1}^m \mathbb{P}(k)\cdot\mathbb{E}[X_k^M] 
    %= R^M\cdot\frac{\Gamma(n+1)}{\Gamma(n+1+\frac{M}{D})}\sum_{k=1}^m \mathbb{P}(k)\cdot\frac{\Gamma(k+\frac{M}{D})}{\Gamma(k)}      
    \label{eq:E_X_general}, 
\end{align}
where $\mathbb{P}(k)$ and $\mathbb{E}[X_k^M]$ are respectively given by: 
\begin{align}
    \mathbb{P}(k) = \frac{1}{m}\left[ \left(\frac{k-1}{n}\right)^{k-1} + \sum_{i=k+1}^m \left(\frac{i-1}{n}\right)^{k-1}\left(1-\frac{i-1}{n}\right) \right],
\end{align}
\begin{align} \label{eq:E_X_k}
    \mathbb{E}[X_k^M] = R^M\cdot\frac{\Gamma(n+1)}{\Gamma(n+1+\frac{M}{D})}\cdot\frac{\Gamma(k+\frac{M}{D})}{\Gamma(k)}.  
\end{align}
Here $R$ represents the radius of the unit-volume hyper-ball, which is given by:
\begin{align} \label{eq:R}
    R = \frac{\left(\Gamma (\frac{D}{p}+1)\right)^{\frac{1}{D}}}{2\Gamma(\frac{1}{p}+1)}, 
\end{align} 
and $I_{z}(a,b) = \frac{\text{B}(z;a,b)}{\text{B}(a,b)}$ is the regularized beta function, $\text{B}(z;a,b) = \int_0^z t^{a-1}(1-t)^{b-1} \text{d}t$ is the incomplete beta function, $\text{B}(a,b)=\int_0^1 t^{a-1}(1-t)^{b-1} \text{d}t = \frac{\Gamma(a)\Gamma(b)}{\Gamma(a+b)}$ is the beta function, and $\Gamma(z) = \int_0^{\infty}t^{z-1}e^{-t} \text{d}t$ is the gamma function.

\subsection{Scaling property}
\label{sec:scaling}
By setting $M = 1$, Equation \eqref{eq:E_X_general} directly provides an estimate of the expected matching distance $\mathbb{E}[X]$ in a unit-volume hyper-ball. % with a fixed number of vertices, by setting $M = 1$. 
We now examine how $\mathbb{E}[X]$ scales with the volume of the matching region under fixed vertex densities (i.e., the number of vertices per unit volume), and show how this scaling behavior varies with the supply-to-demand ratio.
%it rapidly converges to a finite constant as the region's volume becomes sufficiently large. 
%While in balanced cases, it asymptotically increase with respect to the region size in low-dimensional spaces. 
%, as long as the densities of demand and supply vertices are not equal.

Now let $m$ and $n$ represent the densities of demand and supply vertices, respectively, and let the volume of the hyper-ball be $V$.
The numbers of demand and supply vertices for matching become $mV$ and $nV$, respectively, and the radius of the hyper-ball becomes $R_V = R V^{\frac{1}{D}}$, where $R$ is given by Equation \eqref{eq:R}.
Substituting these values into Equation \eqref{eq:E_X_general}, the %formula for estimating the 
expected matching distance becomes:
\begin{align} \label{eq:E_X_V}
    \mathbb{E}[X] 
    = \frac{RV^{\frac{1}{D}}\Gamma(nV+1)}{mV\Gamma(nV+1+\frac{1}{D})}\sum_{i=1}^{mV} \left[\sum_{k=1}^i \left(\frac{i-1}{nV}\right)^{k-1}\left(1-\frac{i-1}{nV}\right)\frac{\Gamma(k+\frac{1}{D})}{\Gamma(k)} + \left(\frac{i-1}{nV}\right)^{i-1}\frac{\Gamma(i+\frac{1}{D})}{\Gamma(i)}\right].
\end{align}
Next, we examine how Equation \eqref{eq:E_X_V} scales with $V$ under varying $n/m$ ratios. 

First, in the asymptotic case when %According to the convergence property identified by \citet{shen_zhai_ouyang_2024}, when the density of supply vertices significantly exceeds the density of demand vertices (i.e., 
$n \gg m$, it is easy to show \citep{shen_zhai_ouyang_2024} that $\mathbb{E}[X]$ converges to the expected nearest-neighbor distance, which is given by: 
\begin{align} \label{eq:E_X_nearest}
\mathbb{E}[X] 
\xrightarrow{n \gg m} RV^{\frac{1}{D}} \cdot \Gamma\left(1 + \frac{1}{D} \right) \cdot (nV)^{-\frac{1}{D}} = R \cdot \Gamma\left(1 + \frac{1}{D} \right) \cdot n^{-\frac{1}{D}}.
\end{align}
%Here the nearest-neighbor matching corresponds to the scenario where 
Intuitively, when $n \gg m$, each demand vertex is highly likely to be matched to its nearest supply vertex, and the influence of other competing demand vertices is negligible.
%After considering the scale of the region volume, we observe from Equation \eqref{eq:E_X_nearest} that, when $n \gg m$
Hence, not surprisingly, Equation \eqref{eq:E_X_nearest} shows that the expected matching distance is independent of both $m$ and $V$.

For general values of $m$ and $n$, it can be shown that, as long as the densities of demand and supply vertices are not (nearly) equal (i.e., $n\gnapprox m$), $\mathbb{E}[X]$ becomes largely independent of $V$. % as $V$ becomes sufficiently large.
To demonstrate this, we approximate $\mathbb{E}[X]$ by simplifying the gamma functions and summations in Equation~\eqref{eq:E_X_V}, %. This approximation, denoted by $\hat{\mathbb{E}}[X]$, is expressed 
as follows:
\begin{align} \label{eq:E_X_poly}
\hat{\mathbb{E}}[X] =
\frac{R}{mn^{\frac{1}{D}}V}\left[1+\sum_{i=2}^{mV} \left(\frac{nV}{i-1}-1\right)\text{Li}_{-\frac{1}{D}}\left(\frac{i-1}{nV}\right)\right],
\end{align}
where $\text{Li}_{s}\left( x \right) = \sum_{k=1}^{\infty} x ^{k}/k^{s}$ is the poly-logarithm function. %The error bound between 
\citet{shen_zhai_ouyang_2024} proved that the approximation error $|\mathbb{E}[X] - \hat{\mathbb{E}}[X]|$ is very small and asymptotically approaches $0$ as $n$ increases. 
As such, $\hat{\mathbb{E}}[X]$ and $\mathbb{E}[X]$ share approximately the same scaling behavior with respect to parameters such as $m$, $n$, $D$ and $V$. 
Then, by analyzing the monotonicity of the term inside the summation of Equation \eqref{eq:E_X_poly}, we can derive both lower and upper bounds for $\hat{\mathbb{E}}[X]$% by approximating the summation with two definite integrals
, as stated in the following proposition. 

\begin{prop} \label{prop:E_X_bound}
    %Given $m\in \mathbb{Z}^+, n\in \mathbb{Z}^+$, %, D\ge2$,
    \begin{align} \label{eq:E_X_bounds}
        \frac{R}{mn^{\frac{1}{D}-1}}\int_0^{\frac{m}{n}-\frac{1}{nV}} \left(\frac{1}{x}-1\right)\operatorname{Li}_{-\frac{1}{D}}\left(x\right)\operatorname{d}x
        \le
        \hat{\mathbb{E}}[X] 
        \le \frac{R}{mn^{\frac{1}{D}-1}}\int_0^{\frac{m}{n}} \left(\frac{1}{x}-1\right)\operatorname{Li}_{-\frac{1}{D}}\left(x\right)\operatorname{d}x.
    \end{align}
\end{prop}
\begin{proof}
Let function $f(i\mid n,V,D)=\left(\frac{nV}{i-1}-1\right)\text{Li}_{-\frac{1}{D}}\left(\frac{i-1}{nV}\right)$ % denote the term 
inside the summation of Equation \eqref{eq:E_X_poly}, such that 
\begin{align} \label{eq:E_X_fi}
\hat{\mathbb{E}}[X] =
\frac{R}{mn^{\frac{1}{D}}V} \sum_{i=1}^{mV} f(i\mid n,V,D).
\end{align}
Lemma \ref{lemma:f_i} in Appendix \ref{app:f_i} shows that $f(i\mid n,V,D)$ is monotonically increasing with respect to $i$ for $\frac{i - 1}{nV} \in [0, 1)$.
Hence, the summation $\sum_{i = 1}^{mV} f(i\mid n,V,D)$ can be bounded using two definite integrals with appropriately chosen limits, as expressed by the 2$^\text{nd}$ and 3$^\text{rd}$ inequalities below:
\begin{align}
    \int_1^{mV}f(i\mid n,V,D) \text{ d}i 
    \le
    \int_0^{mV}f(i\mid n,V,D) \text{ d}i 
    %\le 
    %\int_0^{mV}f(i\mid n,V,D) \text{ d}i
    \le    
    \sum_{i=1}^{mV} f(i\mid n,V,D)
    \le 
    \int_1^{mV+1}f(i\mid n,V,D) \text{ d}i.
\end{align}
The first inequality clearly holds because the integrand is nonnegative. Then, Equation \eqref{eq:E_X_bounds} is obtained by substituting $x = \frac{i - 1}{nV}$ and merging the above inequalities into Equation \eqref{eq:E_X_fi}. %, we can obtain the inequalities presented in .
\end{proof}

Based on Proposition \ref{prop:E_X_bound}, we can establish the following scaling properties of $\hat{\mathbb{E}}[X]$ with respect to $V$:
\begin{itemize}
    \item[(a)] As $V \rightarrow \infty$, it is easy to see that the upper limits of both the lower and upper bounds of $\hat{\mathbb{E}}[X]$ in Equation \eqref{eq:E_X_bounds} quickly converge, and hence so should $\hat{\mathbb{E}}[X]$; i.e., %. As a result, the expected matching distance converges to the upper bound as follows. 
    \begin{align}
    \hat{\mathbb{E}}[X]
    \xrightarrow{V\to\infty}
    \frac{R}{mn^{\frac{1}{D}-1}}\int_0^{\frac{m}{n}} \left(\frac{1}{x}-1\right)\text{Li}_{-\frac{1}{D}}\left(x\right)\text{ d}x.
    \end{align}
\item[(b)] When $n \gnapprox m$, the above upper bound is a finite constant independent of $V$. This is because the expected distance decreases monotonically with the spatial dimension $D$. Therefore, for $D \geq 1$, we have:
\begin{align}
\hat{\mathbb{E}}[X] \le 
\frac{R}{mn^{\frac{1}{D}-1}}\int_0^{\frac{m}{n}} \left(\frac{1}{x}-1\right)\text{Li}_{-\frac{1}{D}}\left(x\right)\text{ d}x 
\le
\frac{R}{mn^{\frac{1}{D}-1}} \int_0^{\frac{m}{n}} \frac{1}{1-x} \text{ d}x
= \frac{-R\ln (1-\frac{m}{n})}{mn^{\frac{1}{D}-1}}.
\end{align}
Notably, the right-hand side of the above equation is unbounded as $n \to m$, but remains finite as long as $n \gnapprox m$.

\item[(c)] 
When $n \gnapprox m$, the convergence rate of $\hat{\mathbb{E}}[X]$ with respect to $V$ can be analyzed by examining the first-order derivative of the lower bound in Equation \eqref{eq:E_X_bounds}, %as shown below:
\begin{align} \label{eq:derivative_V}
%h(m,n,D,V) = 
\frac{R}{mn^{\frac{1}{D}-1}} \cdot \left(\frac{1}{x}-1\right)
\text{Li}_{-\frac{1}{D}}\left(x\right) \cdot \frac{1}{nV^2},
\end{align}
where $x = \frac{m}{n}-\frac{1}{nV} < \frac{m}{n} \le 1$.
First, it is easy to verify that the derivative is non-negative for $V\in (0,+\infty)$. Second, as shown in Figure \ref{fig:f_x} in Appendix \ref{app:f_i}, the function $\left(\frac{1}{x} - 1\right) \operatorname{Li}_{-\frac{1}{D}}(x)$ stays almost constant around $1$ for $x \lnapprox 1$, and only begins to increase significantly as $x\to 1$.
As such, when $n \gnapprox m$, the derivative monotonically decreases with $V\in (0,+\infty)$ and approaches zero at a rate of $\mathcal{O}(1/V^2)$.
These suggest that the lower bound of $\hat{\mathbb{E}}[X]$ increases rapidly at small values of $V$ and approaches its upper limit %of $\hat{\mathbb{E}}[X]$) 
when $V$ becomes sufficiently large. 
%This suggests that the expected matching distance rapidly approaches its upper bound as $V$ increases. %becomes sufficiently large.
%Moreover, the rate of convergence increases with both the dimensionality $D$ and the supply-to-demand ratio $\frac{n}{m}$.
%\item[(d)]
%In addition, given any fixed $m$, $n$, and $V$, it is straightforward to verify that 

%the derivative of Equation \eqref{eq:E_X_bounds} with respect to the spatial dimension $D\in \mathbb{Z}^+$, 
%$$
%\frac{\partial}{\partial D} h(m,n,D,V) =
%\frac{R}{m n V^2} \cdot \left( \frac{1}{x} - 1 \right)
%\cdot \left[
%\frac{1}{D^2} \ln n \cdot n^{1 - \frac{1}{D}} \cdot \operatorname{Li}_{-\frac{1}{D}}(x)
%+
%n^{1 - \frac{1}{D}} \cdot \frac{1}{D^2} \cdot \frac{\partial}{\partial D} %\operatorname{Li}_{-\frac{1}{D}}(x) 
%\right].
%$$
%is positive. This indicates that $h(m,n,D,V)$ grows monotonically with $D\in[1,+\infty)$, 
%implying that the convergence rate of $\hat{\mathbb{E}}[X]$ to its upper bound is faster in higher-dimensional spaces. 
\end{itemize}

In summary, the scaling properties (a)–(c) indicate that, for a given spatial region with fixed vertex densities $m$ and $n$, when $n \gnapprox m$, the expected matching distance becomes largely independent of the region's volume $V$, and quickly converges to a finite constant as $V$ increases.
However, for the more balanced cases where $n \gtrapprox m$, the scaling behavior could differ from that of the unbalanced cases.
As identified in property (a), the upper bound of the expected matching distance may become unbounded as $V \to \infty$. 
Specifically, when $m = n$, prior studies (e.g., \citet{caracciolo_scaling_2014}) have shown that (i) the expected matching distance scales with the region volume and goes to infinity as $V \to \infty$ for $D = 1$ and $D = 2$, while (ii) the distance converges to a finite constant for $D \ge 3$. 
%These findings are also consistent with property (d).
%all terms associated with $D$ in $\mathbb{E}[X]$ in Equation \eqref{eq:E_X_V}, i.e., $R, \frac{\Gamma(k+\frac{1}{D})}{\Gamma(n+\frac{1}{D}+1)}, \frac{\Gamma(i+\frac{1}{D})}{\Gamma(n+\frac{1}{D}+1)}$ all increase with $D$, and hence $\mathbb{E}[X]$ should increase monotonically with $D$.

These analytical findings will be further verified by the numerical results in Section \ref{sec:numerical}. 
They also imply that, except for the exactly balanced cases, optimal matching in RBMP primarily occur among local neighbors, and hence the local vertex densities dictate the expected optimal matching distance. %result across the entire region. 
This insight serves as a foundation for extending the analysis to heterogeneous RBMPs.

\subsection{Impacts of matching radius}
%To pave the way for our analysis of heterogeneous RBMPs, we first 
We now examine how the optimal matching distance of static homogeneous RBMPs will be further affected by imposing a maximum allowable matching radius. 
As illustrated in Figure \ref{fig:radius and pooling}, this radius truncates the matching distance at a specified threshold, which helps prevent excessively long deadheading for demand/supply vertices. 
%-- in the context of mobility systems, this imposes a limit on the expected supply turnaround time in ST-RBMP and potentially mitigates WGC. However, some demand vertices may remain unmatched at the current decision epoch and be left over to the later ones. As such, the radius must be carefully chosen to balance the competing matching objectives, as will be discussed later in Section \ref{sec:RBMP_dynamic}. 
This section focuses on adapting Equations \eqref{eq:F_X_general} and \eqref{eq:E_X_general} 
%to estimate the resulting matching probability and expected distance 
under such truncation. 

%Consider a static and homogeneous RBMP with vertex densities $m$ and $n$ within a unit-volume hyper-ball. 
Let $r \in [0, 1]$ denote the matching radius as a proportion of the hyper-ball's radius $R_V = R V^{\frac{1}{D}}$. Any pair of vertices is considered infeasible for matching if their distance exceeds $rR_V$.
Let tuple $\chi = (m, n, r, V)$ denote the key parameters that determine the matching outcomes under this setting.
We focus on quantifying two key metrics: the successful matching probability, which represents the proportion of demand vertices that are successfully matched, denoted by $p(\chi)$, and the expected matching distance per successfully matched vertex, denoted by $d(\chi)$. 
%Based on the scaling properties, $V$ primarily  affects these metrics in balanced cases.

%Note that in the trivial case where $r = 1$ (i.e., no matching radius is imposed), we have $p(m,n,r) = 1$, and $d(m,n,r) = \mathbb{E}[X]$ as given by Equation \eqref{eq:E_X_general}. 
For general values of $r \in [0,1]$, the two-step approach introduced at the beginning of Section \ref{sec:RBMP_static} for deriving Equations \eqref{eq:F_X_general} and \eqref{eq:E_X_general} still applies. In particular, the probability of matching to one's $k$-th nearest neighbor, $\mathbb{P}(k)$, in step (i) remains unchanged. Yet, the distribution of the matching distance to the $k$-th nearest neighbor in step (ii) is truncated by $X_k \leq rR_V$.
Therefore, the overall matching probability $p(\chi) = \Pr\{X\le rR_V\}$ can be estimated by the value of CDF of $X$ %from Equation \eqref{eq:F_X_general} 
at $x = rR_V$, as follows:
\begin{align} \label{eq:p_mnr}
p(\chi)
\approx \sum_{k=1}^{mV} \mathbb{P}(k) \cdot F_{X_k}(rR_V)
= \sum_{k=1}^{mV} \mathbb{P}(k) \cdot I_{r^D}(k, nV - k + 1).
\end{align}
In addition, the $M$-th moment of the matching distance under such a truncation, $\mathbb{E}[X^M \mid X \le rR_V]$, can be derived by replacing %the $M$-th moment of the $k$-th distance $\mathbb{E}[X_k^M]$ in 
Equation \eqref{eq:E_X_general} by:
$$\mathbb{E}[X_k^M \mid X_k \le rR_V] = \frac{\int_0^{rR_V} x^M\, \mathrm{d}F_{X_k}(x)}{F_{X_k}(rR_V)}.$$
As a result, we have: 
\begin{align} \label{eq:E_X_moment}
\begin{split}
%d(m,n,r)
\mathbb{E}[X^M \mid X \le rR_V] 
&\approx
\sum_{k=1}^{mV} \mathbb{P}(k) \cdot \mathbb{E}[X_k^M \mid X_k \le rR_V] 
= \sum_{k=1}^{mV} \mathbb{P}(k) \cdot
\frac{\int_0^{rR_V} x^M \,\mathrm{d}I_{\left(\frac{x}{R_V}\right)^D}(k,nV-k+1)}{I_{r^D}(k,nV-k+1)} \\
& = R^M V^{\frac{M}{D}} \cdot \sum_{k=1}^{mV} \mathbb{P}(k) \cdot
\frac{\mathrm{B}(r^D; k + \frac{M}{D}, nV - k + 1)}{\mathrm{B}(r^D; k, nV - k + 1)}.
\end{split}
\end{align}

It is easy to verify that when $r = 1$ (i.e., effectively, no matching radius is imposed), we have $p(\chi) = 1$, and Equation \eqref{eq:E_X_moment} reduces to Equation \eqref{eq:E_X_general}. The truncated expectation $d(\chi)$ is obtained by simply taking $M = 1$; i.e.,
\begin{align} \label{eq:d_mnr}
d(\chi) = R V^{\frac{1}{D}} \cdot \sum_{k=1}^{mV} \mathbb{P}(k) \cdot
\frac{\mathrm{B}(r^D; k + \frac{1}{D}, nV - k + 1)}{\mathrm{B}(r^D; k, nV - k + 1)}.
\end{align} 
Moreover, the variance of the corresponding optimal matching distance, $\mathbb{V}[X \mid X \le rR_V]$, can be computed from the first two moments as follows:
\begin{align} \label{eq:Var_X}
\mathbb{V}[X \mid X \le rR_V] = \mathbb{E}[X^2 \mid X \le rR_V] - \left(\mathbb{E}[X \mid X \le rR_V]\right)^2.
\end{align}

\section{Static RBMPs under Spatial Heterogeneity}
\label{sec:RBMP_spatial_heter}

Now we are ready to extend these analytical results to static and heterogeneous RBMPs, where the densities of demand and supply vertices in the region vary across locations. To effectively address the spatial heterogeneity, we assume that the region can be partitioned into a set $Z=\{1,\ldots,|Z|\}$ of smaller zones, indexed by $z$. 
%During the partition, it is recommended to cluster the areas with similar demand and supply densities into a single zone and to keep t
Each zone should be approximately round (e.g., square or hexagon in two dimensions, cube or hexagonal prism in three dimensions) to mimick a hyperball, and large enough to contain at least a few vertices from each subset.\footnote{Whenever appropriate, a few larger zones are preferred over many smaller zones.} % than } size as large as possible. Additionally, each zone should have an approximately round shape (e.g., square or hexagon), such that its radius can be approximated by that of a hyper-ball of the same size in the same dimension. 
%In our problem setting, t
The sizes of zone $z \in Z$ %after partitioning is provided as input and 
is denoted by $V_z$, and both demand and supply vertices inside this zone are %assumed to be approximately homogeneous, and the corresponding vertices are 
generated independently from homogeneous Poisson processes with mean densities $m_z$ and $n_z$, %$\forall z\in Z$, 
respectively. %\footnote{
%Here a piecewise constant approximation of the vertex densities is adopted, rather than assuming each local neighborhood to be part of an infinite plane.
%This is because, in our derivation, both the matching distance and the matching radius are proportional to the radius of the region of analysis, while the region's boundary plays a non-negligible role. In particular, in balanced RBMPs, the expected matching distance asymptotically increase with respect to the region size in low-dimensional spaces (e.g., $D = 1$ and $D = 2$)
%when no matching radius is imposed. 
%}
In this paper, to stay focused, we only discuss the case where the mean density of supply vertices is always larger or equal to that of the demand vertices in all zones; i.e., $n_z \ge m_z, \forall z\in Z$. 
Figure \ref{fig:RBMP_heter_profile} illustrates an example of a two-dimensional region partitioned into a set of hexagonal zones, where the demand density distribution is shown in the heatmap. 
For each realized problem instance, the matching is performed over the entire region while each demand vertex is subject to a zone-specific maximum matching radius $r_z, \forall z\in Z$. % imposed for its respective zone. 
An example of realized vertex locations and the corresponding optimal matching solution is shown in Figure \ref{fig:RBMP_heter_instance}. 
Unmatched demand vertices are enclosed by shaded circles, representing their respective matching radii, within which no unmatched supply vertices are available.
\begin{figure}[ht!]
    \centering
    \begin{subfigure}{0.49\textwidth}
        \centering
        \includegraphics[width=0.8\textwidth]{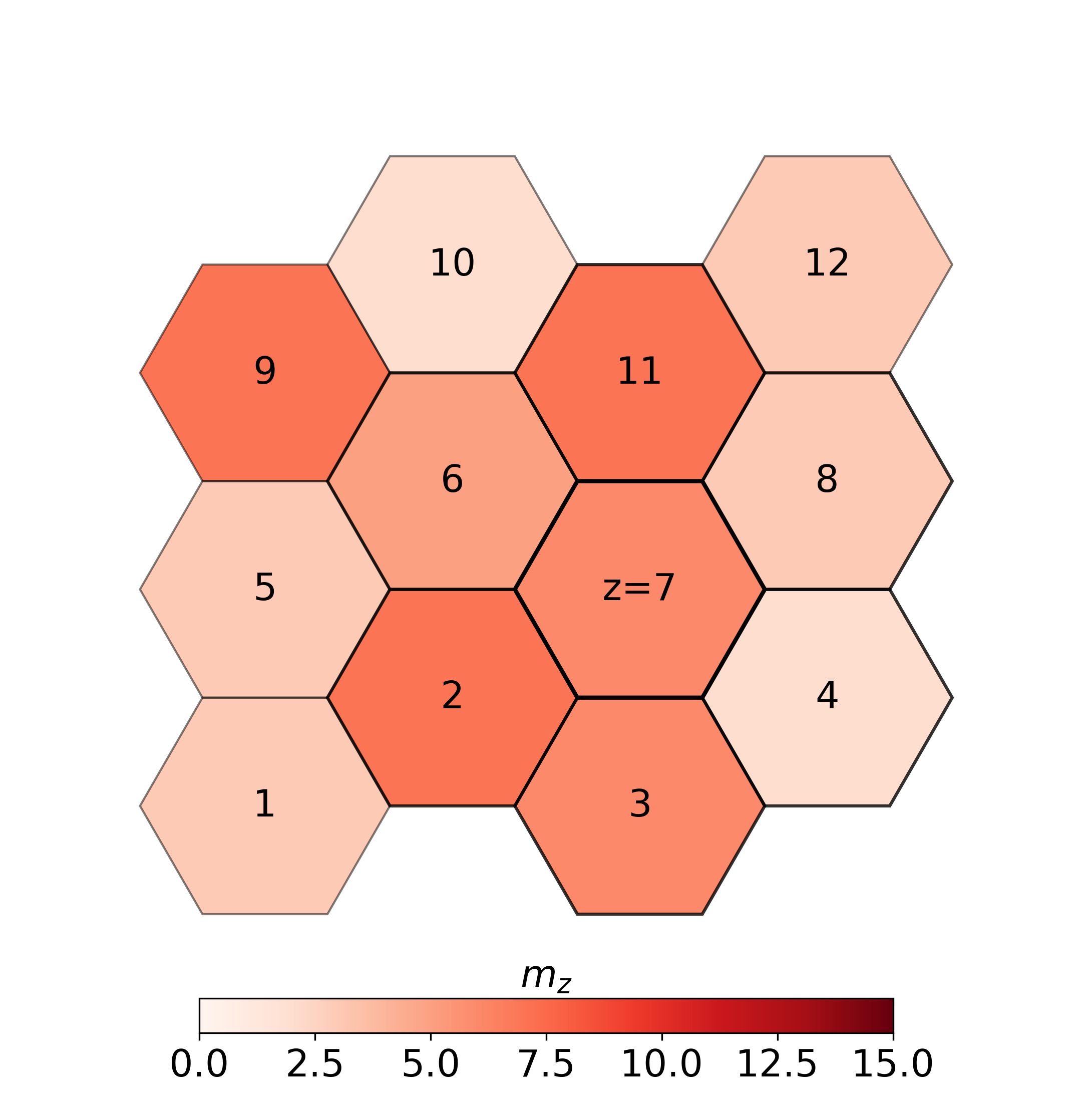}
        \caption{Example Density Distribution.}
        \label{fig:RBMP_heter_profile}
    \end{subfigure}
    %\hfill
    %\bigskip
    \begin{subfigure}{0.49\textwidth}
        \centering
        \includegraphics[width=0.8\textwidth]{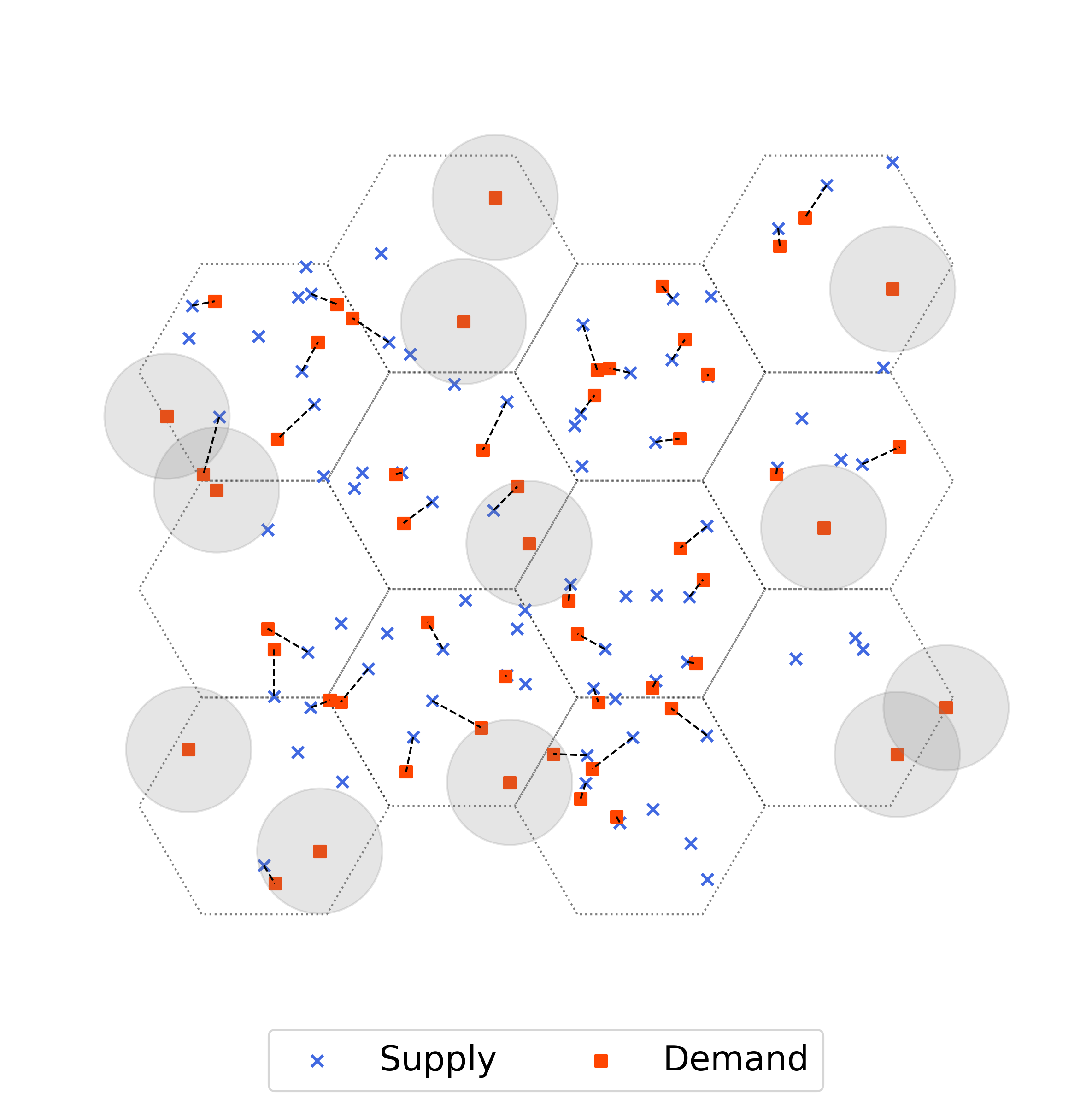}
        \caption{Example Problem Instance.}
        \label{fig:RBMP_heter_instance}
    \end{subfigure}
    \caption{RBMP under Spatial Heterogeneity.}
    \label{fig:RBMP_heter}
\end{figure}

%the sets $\mathbf{m} = \{ m_z \mid z \in Z \}$, $\mathbf{n} = \{ n_z \mid z \in Z \}$, and $\mathbf{r} = \{ r_z \mid z \in Z \}$ 
Let tuple $\chi_z = (m_z, n_z, r_z, V_z)$ represents the local parameter profile for each zone $z \in Z$, including the mean demand density, mean supply density, matching radius, and zone size, and let tuple $\boldsymbol{\chi} = (\chi_1,\ldots,\chi_{|Z|})$ represents the collective parameter profile across all zones, which captures the spatial heterogeneity in a given RBMP. 
%Given this profile, o
Our objective is to estimate the matching probability and expected matching distance per demand vertex: (i) in each zone $z \in Z$, denoted by $p_z(\boldsymbol{\chi})$ and $d_z(\boldsymbol{\chi})$, respectively; and (ii) in the entire region, denoted by $\bar{p}(\boldsymbol{\chi})$ and $\bar{d}(\boldsymbol{\chi})$, respectively.

To estimate these metrics, we first adapt the approach proposed by \citet{zhai_average_2024}, which was originally developed to estimate the expected matching distance in a discrete regular network, where vertices are distributed along homogeneous one-dimensional network edges and distances are measured along the shortest path. %However, their model only considers homogeneous edges with the same length and vertex distributions. 
Here, we model the zones in an entire region as connected ``edges" in a heterogeneous network, each with a %where the partitioned zones correspond to the connected ``edges" with 
varying size and vertex distributions.
From the perspective of a demand vertex in a specific zone $z \in Z$, a matching result can occur as one of two types:
(i) a ``local" match, where the matched supply vertex is within the same zone, with expected distance $d_z^{\text{l}}(\boldsymbol{\chi})$; and
(ii) a ``global" match, where the matched supply vertex is not within the same zone, with expected distance $d_z^{\text{g}}(\boldsymbol{\chi})$.
Let $\alpha_z$ denote the probability for a demand vertex in zone $z$ to have a global match. By the law of total expectation, the expected matching distance $d_z(\boldsymbol{\chi})$ can be expressed as:
\begin{align}
\label{eq:d_z_profile}
d_z(\boldsymbol{\chi}) = (1 - \alpha_z) \cdot d_z^{\text{l}}(\boldsymbol{\chi}) + \alpha_z \cdot d_z^{\text{g}}(\boldsymbol{\chi}).  
\end{align}

%According to \citet{zhai_average_2024}, the key to estimating 
Quantities $d_z^{\text{l}}(\boldsymbol{\chi})$, $\alpha_z$, and $d_z^{\text{g}}(\boldsymbol{\chi})$ can be %in Equation \eqref{eq:d_z_profile} is based on 
estimated via a heuristic matching process. 
%that prioritizes local matches. 
%From the perspective of a specific zone $z \in Z$, i
If the realized number of supply vertices exceeds that of demand vertices, all demand vertices are matched locally as if the zone were isolated. The number of excessive supply vertices of zone $z$ is denoted by a random variable $n_z^+$. 
Otherwise, if the realized demand exceeds supply, all supply vertices are prioritized to be locally matched with the demand vertices located closer to the center of the zone, while the excessive demand vertices, with a total number of $m_z^+$, will seek matches globally. %From the perspective of an excessive demand vertex, a 
A global match is sought through a breadth-first-search (BFS) procedure (based on adjacency) across the zones surrounding $z$. Let $Z_z^k$ be the set of zones that are reachable from $z$ in exactly the $k$-th layer. %steps in the BFS following a tree-like structure. 
For example, in Figure \ref{fig:RBMP_heter_profile}, the set of zones in the first layer of zone $z=7$ is $Z_7^1 = \{2, 3, 4, 6, 8, 11\}$, while the set of the second layer is $Z_7^2 = \{1, 5, 9, 10, 12, \cdots\}$. 
The excessive demand vertex is matched to one of the excessive supply vertices found in the nearest available layer.
Based on this process, $d_z^{\text{l}}(\boldsymbol{\chi})$, $\alpha_z$, and $d_z^{\text{g}}(\boldsymbol{\chi})$ can be estimated as follows.
\begin{itemize}
    \item[(i)] The local matching distance $d_z^{\text{l}}(\boldsymbol{\chi})$ can be effectively approximated by treating the matching within each zone as a homogeneous RBMP, given the local parameter profile $\chi_z = (m_z, n_z, r_z, V_z)$; i.e.,
    \begin{align} \label{eq:d_z_local}
    d_z^{\text{l}}(\boldsymbol{\chi}) \approx d(\chi_z),
    \end{align}
where $d(\chi_z)$ is given by Equation \eqref{eq:d_mnr}. 
\item[(ii)] The global matching probability $\alpha_z$ can be estimated as the expected fraction of globally matched demand vertices in zone $z$ as the following:
\begin{align}
\label{eq:alpha_z}
\alpha_z 
\approx \frac{\mathbb{E}[m_z^+ ]}{m_z}
=\frac{1}{m_z V_z} \cdot \text{Pr}\{m_z^+ > 0\} \cdot \mathbb{E}[m_z^+ \mid m_z^+>0]
,
\end{align}
where %$\text{Pr}\{m_z^+ > 0\}$ denotes the probability that zone $z$ has excessive demand vertices, and 
$\mathbb{E}[m_z^+ \mid m_z^+>0]$ denotes the conditional expectation of the density of excessive demand vertices. They can be estimated by approximating the distribution of $m_z^+$, which is the difference between two Poisson random variables (with means $m_z$ and $n_z$), by a normal distribution, as follows.
\begin{align}
\text{Pr}\{m_z^+ > 0\} &\approx \Phi\left(\frac{-\frac{1}{2}+(m_z-n_z)V_z}{\sqrt{(n_z+m_z)V_z}}\right), \label{eq:p_m+}\\
%\end{align}
%\begin{align} 
\mathbb{E}[m_z^+ \mid m_z^+>0] &\approx %\left\{
(m_z-n_z)V_z + \sqrt{(n_z+m_z)V_z} \cdot \frac{ \phi\left( \frac{-\frac{1}{2}+(n_z - m_z)V_z}{\sqrt{(n_z+m_z)V_z}} \right)}{1-\Phi \left( \frac{-\frac{1}{2}+(n_z - m_z)V_z}{\sqrt{(n_z+m_z)V_z}} \right)}.
\label{eq:E_m+}
%\right\}
\end{align}
Here $\phi(\cdot)$ and $\Phi(\cdot)$ are the probability density function (PDF) and the CDF of the standard normal distribution, respectively. 
\item[(iii)] The global matching distance $d_z^{\text{g}}(\boldsymbol{\chi})$ consists of three legs: 
(a) the intra-zone distance from an excessive demand vertex to the boundary of its ``origin" zone $z$;
(b) the inter-zone distance from the boundary of the origin zone to the boundary of the ``destination" zone that contains the matching point; % a global match is identified via the BFS procedure;
(c) the intra-zone distance from the boundary of the destination zone to the match point.

Among the three legs, leg (b) is directly related to the probability of finding a global match in a zone in the $k$-th layer $Z_z^k$. This probability can be computed as the likelihood that an excessive supply vertex is successfully found in the $k$-th layer, but not in any of the previous $k-1$ layers, as follows:
\begin{align} \label{eq:p_k_layer}
\left( 1-\prod_{z\in Z_z^k}\Pr\{m_z^+ > 0\}\right) \cdot \prod_{i=0}^{k-1}\prod_{z\in Z_z^i}\Pr\{m_z^+ > 0\}.
\end{align}
%In a heterogeneous RBMP, it is likely that a given zone is surrounded by some neighboring zones with unbalanced demand and supply; i.e., 
Since we assume that $n_z \ge m_z$ in all zones, the probability of having excessive demand, $\Pr\{m_z^+ > 0\}$, is likely small. Additionally, %following the layered structure in BFS, the number of zones in the $k$-th layer, $|Z_z^k|$, increases rapidly with $k$. 
%As a result, the probability of finding a global match at the $k$-th layer 
the product term in Equation \eqref{eq:p_k_layer} shall diminish rapidly to $0$ as $k$ increases. This indicates that a global match is highly likely to be found in the first few layers. 
Furthermore, if a matching radius is imposed, it further restricts the global matches to be found in nearby zones. 
As such, we simplify the analysis by assuming that all global matches are found in a zone within the first layer, $z' \in Z_z^1$; % (i.e., the nearest neighboring zones); % while ignoring global matches in more distant layers, 
i.e.,  
$d_z^{\text{g}}(\boldsymbol{\chi})$ can be approximated by using only legs (a) and (c). 

%By adapting the formulas in \citet{zhai_average_2024} to account for spatial heterogeneity, and from one dimension to higher dimensions, 
Leg (a) here can be approximated by the average expected shortest distance from each excessive demand vertex in zone $z$ to the boundary of $z$.
Under the proposed heuristic matching procedure, we expect $m_z V_z$ demand vertices in zone $z$, and the excessive ones are located farthest from the zone center. The expected number of these excessive vertices is $\mathbb{E}[m_z^+] V_z$.
For $k$-th nearest vertex, its distance to the center can be estimated by %the $k$-th nearest neighbor distance formula from 
Equation \eqref{eq:E_X_k}, and $k$ ranges from $m_z V_z - \mathbb{E}[m_z^+] V_z + 1$ to $m_z V_z$ for the excessive vertices. The corresponding distance to the boundary is the difference between the radius of $z$, $R_{V_z}$, and the distance to the center. 
Taking average across all excessive demand vertices gives the first term in Equation \eqref{eq:d_z_global}. The analysis on leg (b) %, we can first compute the average expected shortest distance from each excessive supply vertex in zone $z'$ to the boundary of $z'$ 
is exactly similar to that on leg (a), which gives %and then take another average of these estimates over all $z' \in Z_z^1$. This is represented by 
the second term in Equation \eqref{eq:d_z_global}. 
\begin{equation}
\begin{aligned}
\label{eq:d_z_global}
d_z^{\text{g}}(\boldsymbol{\chi}) 
\approx 
%\frac{1}{2}\cdot\frac{R_{V_z}}{m_z} \cdot \mathbb{E}[m_z^+ \mid m_z^+ > 0]
&\frac{1}{\mathbb{E}[m_z^+]V_z} \cdot
\sum_{k=m_zV_z-\mathbb{E}[m_z^+]V_z+1}^{m_zV_z} R_{V_z} \left(1-%R_{V_z}\cdot
\frac{\Gamma(m_zV_z+1)}{\Gamma(m_zV_z+1+\frac{1}{D})}\cdot\frac{\Gamma(k+\frac{1}{D})}{\Gamma(k)} \right)\\
&+ 
\frac{1}{|Z_z^1|}\sum_{z'\in Z_z^1} \frac{1}{\mathbb{E}[n_z^+]V_z}\cdot
\sum_{k=n_zV_z-\mathbb{E}[n_z^+]V_z+1}^{n_zV_z} R_{V_{z'}} \left(1-%R_{V_{z'}}\cdot
\frac{\Gamma(n_zV_z+1)}{\Gamma(n_zV_z+1+\frac{1}{D})}\cdot\frac{\Gamma(k+\frac{1}{D})}{\Gamma(k)} \right).
%\sum_{z'\in Z_z^1} \frac{1}{2|Z_z^1|}\cdot\frac{m_{z'}R_{V_{z'}}}{n_{z'}^2} \cdot \mathbb{E}[n_{z'}^+ \mid n_{z'}^+>0]. 
\end{aligned}
\end{equation}
%The term $\frac{R_{V_z}}{m_z}$ corresponds to the expected thickness the hyper-spherical shell near the boundary of the original zone in leg (a), where the excessive demand vertices are expected to be located. 
%The term $\frac{m_{z'}R_{V_{z'}}}{n_{z'}^2}$  corresponds to the expected thickness the hyper-spherical shell near the boundary of the destination zone in leg (c), where the excessive supply vertices are expected to be located. %first term represents the expected distance from a randomly selected excessive demand vertex to the boundary of zone $z$, while the second term represents the average expected distance from a randomly selected excessive supply vertex to the boundary of zone $z'$, $\forall z’ \in Z_z^1$.
Here $\mathbb{E}[n_{z'}^+] = \text{Pr}\{n_{z'}^+ > 0\} \cdot \mathbb{E}[n_{z'}^+ \mid n_{z'}^+ > 0]$ denotes the expected density of excessive supply vertices at zone $z'$, where the corresponding probability and conditional expectation $\text{Pr}\{n_{z'}^+ > 0\}$ and $\mathbb{E}[n_{z'}^+ \mid n_{z'}^+ > 0]$ can be computed similarly to Equations \eqref{eq:p_m+} and \eqref{eq:E_m+}, as follows:
\begin{align} 
\text{Pr}\{n_{z'}^+ > 0\} %\text{Pr}\{m^+ > 0\} = 
&\approx \Phi\left(\frac{-\frac{1}{2}+(n_z-m_z)L}{\sqrt{(m_z+n_z)L}}\right),\\
\mathbb{E}[n_{z'}^+ \mid n_{z'}^+>0] 
&\approx 
(n_{z'}-m_{z'})V_{z'} + \sqrt{(n_{z'}+m_{z'})V_{z'}} \cdot \frac{ \phi\left( \frac{-\frac{1}{2}+(m_{z'}-n_{z'})V_{z'}}{\sqrt{(n_{z'}+m_{z'})V_{z'}}} \right)}{1-\Phi \left( \frac{-\frac{1}{2}+(m_{z'}-n_{z'})V_{z'}}{\sqrt{(n_{z'}+m_{z'})V_{z'}}} \right)}.\label{eq:E_n+}
\end{align}
\end{itemize}
By combining Equations \eqref{eq:d_z_local}, \eqref{eq:alpha_z} and \eqref{eq:d_z_global} into Equation \eqref{eq:d_z_profile}, we obtain the expected matching distance $d_z(\bm{\chi})$ for each zone $z\in Z$.

It can be seen that the computation of $d_z(\boldsymbol{\chi})$ in Equation \eqref{eq:d_z_profile} involves evaluating a set of probabilities derived from normal distributions and requires the collective parameter profile for all zones in $Z_z^1$. This formula would be particularly accurate, but it could be computationally cumbersome as well --- especially if the formulas must be embedded % estimates in static and heterogeneous RBMPs.
%While if we aim to incorporate the estimates 
into other optimization or equilibrium modeling frameworks, as the case in Section \ref{sec:RBMP_dynamic}. Hence, we further propose a simpler yet effective approximation below. 

Based on the scaling properties discussed in Section \ref{sec:scaling}, optimal matching in RBMP primarily occurs among local neighbors (especially for unbalanced cases). In a discrete network, \citet{zhai_average_2024} also found that the local matching distance closely approximates the overall matching distance on a one-dimensional edge, when demand and supply distributions are unbalanced and/or when each edge has a reasonably large number of neighbors (e.g., $\ge 5$), so that  %These conditions reflect the scenarios where the excessive demand vertices are likely to find 
global matches are likely to be found within very nearby zones.
These observations suggest that, in our problem setting, where each zone has a sufficiently large number of neighbors (e.g., when $D \ge 2$, a hexagonal zone in a two-dimensional region typically has at least $6$ neighboring zones) and the surrounding zones have unbalanced demand and supply due to heterogeneity, the global matching distance can be approximated by local matching distance within the zone; that is, $d_z(\boldsymbol{\chi}) \approx d_z^{\text{l}}(\boldsymbol{\chi})$. 
%Second, based on the scaling properties discussed in Section \ref{sec:scaling}, the optimal matching distance is primarily determined by local vertex densities. This implies that the local matching distance $d_z^{\text{l}}(\boldsymbol{\chi})$ can be accurately approximated using only the local profile $\chi_z$ from homogeneous RBMPs; i.e., $d_z^{\text{l}}(\boldsymbol{\chi}) \approx d(\chi_z)$. 
As such, we propose to estimate both the expected matching distance and matching probability in each zone $z$ %of a heterogeneous RBMP 
by directly applying the results obtained from homogeneous RBMPs based on the local parameter profile $\chi_z$, as follows.
\begin{equation} \label{eq:d_p_z_approx}
    d_z(\bm{\chi}) \approx d(\chi_z), \quad p_z(\bm{\chi}) \approx p(\chi_z), \quad \forall z\in Z,
\end{equation}
In addition, the overall matching estimates for the entire region equals the weighted average of all the zone-specific local estimates: 
\begin{equation} \label{eq:P and D}
\begin{aligned}
\bar{d}(\bm{\chi}) \approx \frac{\sum_{z\in Z} m_z \cdot d(\chi_z)}{\sum_{z\in Z} m_z},
\quad
\bar{p}(\bm{\chi}) \approx \frac{\sum_{z\in Z} m_z \cdot p(\chi_z)}{\sum_{z\in Z} m_z}.
\end{aligned}
\end{equation}
These approximations will be further verified by the numerical experiments in Section \ref{sec:numerical}.

\section{Dynamic RBMP under Spatiotemporal Heterogeneity}
\label{sec:RBMP_dynamic}
In this section, we develop a dynamic modeling framework for heterogeneous ST-RBMP. %and demonstrate its application 
%in the context of mobility service in a two-dimensional space. 
In order to be specific on the system dynamics and control actions, we use the mobility service in a two-dimensional space (see Figure \ref{fig:radius and pooling}) as an example of application contexts.\footnote{This modeling framework is applicable to many other location-based problems, such as on-demand parcel delivery and emergency resource allocation.} %to illustrate the system dynamics. As shown , e
Each demand vertex represents a customer requesting for service from an origin to a destination, while each supply vertex represents a vehicle available to perform the service. 
New customers continue to enter the system according to given demand arrival patterns. 
The system periodically decides when and how to perform matching between customers and vehicles based on their current locations. 
Each customer experiences a waiting (for matching) cost as the time elapsed between the customer’s service request and its successful matching with a vehicle.
Once a match is made, the assigned vehicle moves toward the corresponding customer at a given speed, incurring a cost for vehicle traveling (deadheading) until customer pickup. 
%, incurring a matching cost equal to the travel time from the assignment to customer pickup. 
Once the vehicle reaches the customer's origin, they will take the customer to its destination. 
After the delivery is completed, the vehicle becomes available (as a newly arrived idle vehicle) again. 
%As the system evolves, the idle vehicles can also be repositioned to different locations to better accommodate with the customer demand. 
In an ``open" system with freelance drivers, the number of vehicles in the system is not constant --- existing idle vehicles may exit, or new idle vehicles may enter the system at any time. 
%The supply arrival patterns can be captured by the given parameters. 

Building upon all results in Sections \ref{sec:RBMP_static}-\ref{sec:RBMP_spatial_heter}, the ST-RBMP is formulated as a dynamic control problem over a planning horizon. % using a continuum approximation scheme. 
To determine the key hyper-parameters for real-time matching in such a system (including pooling intervals and matching radii) that can minimize the system-wide costs
The control variables include key hyper-parameters for real-time matching, including pooling intervals and matching radii, and the goal is to minimize the overall system-wide costs experienced by vehicles and customers. %trajectories of these hyper-parameters are solved based on the optimality conditions. This framework is also applicable to other general location-based problems, such as on-demand parcel delivery and emergency resource allocation.

\subsection{Problem formulation}
Consider a given region of analysis in a $D$-dimensional L$^p$ space.
The units for distance and time are denoted as du and tu, respectively. 
The densities of demand and supply vertices (i.e., customers and idle vehicles) vary slowly over both time and space. As defined in Section \ref{sec:RBMP_spatial_heter}, the region can be partitioned into a set $Z=\{1,\ldots,|Z|\}$ of zones, each with approximately homogeneous vertex distributions. The size/volume of zone $z\in Z$ is $V_z$ [du$^D$]. 
The temporal planning horizon is $[0, \mathcal{T}]$ [tu], during which new demand and supply vertices are generated independently from homogeneous Poisson processes within each zone $z \in Z$, following time-dependent rate functions $\lambda_z(t)$ [\#/du$^D$-tu] and $\mu_z(t)$ [\#/du$^D$-tu], respectively. At any time $t \in [0, \mathcal{T}]$, the mean densities of demand and supply vertices in each zone $z \in Z$ are denoted by $m_z(t)$ [\#/du$^D$] and $n_z(t)$ [\#/du$^D$], respectively. The initial densities at $t = 0$, $m_z(0)$ and $n_z(0)$, $\forall z\in Z$, are assumed to be known. 
Again, we only discuss the case where the density of supply vertices is always larger than that of the demand vertices; i.e., $n_z(t) \ge m_z(t), \forall z\in Z, t\in[0, \mathcal{T}]$. 
A set of matching decision epochs $\{t_i \,|\, i =0, 1, \dots \}$, shared across all zones, and a set of zone-specific matching radii at each epoch $\{r_z(t_i) \,|\, i =0, 1, \dots; z \in Z \}$, need to be determined jointly. 
The pooling interval is the time separation between two consecutive decision epochs, $\tau(t_i) = t_{i+1} - t_i$, for $i =0, 1, \dots$ 

At each decision epoch $t_i$, a static and heterogeneous RBMP instance is solved. The system dynamics and costs between two consecutive decision epochs can be formulated as follows.
Proper units are chosen for du and tu, such that the average vehicle travel speed can be $1$ [du/tu], and the average vehicle deadheading time equals the expected matching distance of the corresponding RBMP.
Based on the results from Section \ref{sec:RBMP_spatial_heter}, for each zone $z\in Z$ at any $t\in [0, \mathcal{T}]$, the matching probability and expected matching distance per demand vertex can be computed via Equation \eqref{eq:d_p_z_approx}, given the local parameter profile $\chi_z(t) = (m_z(t), n_z(t), r_z(t), V_z)$, as follows.
\begin{equation} \label{eq:p and d}
    \begin{aligned}
    & \resizebox{.93\hsize}{!}{$
     p[\chi_z(t)] \approx 
        \frac{1}{m_z(t)V_z} \sum_{i=1}^{m_z(t)V_z} \left[ \sum_{k=1}^{i} \left(\frac{i-1}{n_z(t)V_z}\right)^{k-1}\left(1-\frac{i-1}{n_z(t)V_z}\right)
        I_{r^D_z(t)}(k,n_z(t)V_z-k+1)
        + \left(\frac{i-1}{n_z(t)V_z}\right)^{i} 
        I_{r^D_z(t)}(i,n_z(t)V_z-i+1)
        \right],
        $}
        \\
    & \resizebox{.95\hsize}{!}{$ 
    d[\chi_z(t)] \approx 
        \frac{\left(\Gamma (\frac{D}{p}+1)\right)^{\frac{1}{D}}}{2\Gamma(\frac{1}{p}+1)m_z(t)V_z^{1-\frac{1}{D}}} \sum_{i=1}^{m_z(t)V_z} \left[ \sum_{k=1}^{i} \left(\frac{i-1}{n_z(t)V_z}\right)^{k-1}\left(1-\frac{i-1}{n_z(t)V_z}\right)
        \frac{\text{B}(r^D_z(t);k+\frac{1}{D},n_z(t)V_z-k+1)}{\text{B}(r^D_z(t);k,n_z(t)V_z-k+1)}
        + \left(\frac{i-1}{n_z(t)V_z}\right)^{i} 
        \frac{\text{B}(r^D_z(t);i+\frac{1}{D},n_z(t)V_z-i+1)}{\text{B}(r^D_z(t);i,n_z(t)V_z-i+1)}
        \right].
        $}
    \end{aligned}
\end{equation}
At the $(i+1)$-th decision epoch $t_{i+1}$, the numbers of demand and supply vertices available for matching include both the newly arrived vertices and the leftover vertices from previous decision epochs. 
Then, the demand and supply distribution at $t_{i+1}$ %, 
%$\begin{pmatrix}
%\mathbf{m}_{t_{i+1}} \\
%\mathbf{n}_{t_{i+1}} 
%\end{pmatrix}$
%$(\mathbf{m}_{t_{i+1}}, \mathbf{n}_{t_{i+1}})$, 
satisfies the following:
\begin{equation}
\begin{aligned}\label{eq:mn_t_i+1}
m_z(t_{i+1}) &
=  \lambda_z(t_i) \tau(t_i) + m_z(t_{i}) - p[\chi_z(t_i)] m_z(t_{i}), % \ge 1/V_z, 
\text{  and}
\\
%\quad
n_z(t_{i+1}) &
= \mu_z(t_i) \tau(t_i) 
+ n_z(t_i) - p[\chi_z(t_i)]
m_z(t_{i}), % \ge 1/V_z,
\quad
\forall z\in Z.
\end{aligned}
\end{equation}
The total cost incurred within $[t_{i}, t_{i+1})$ includes three parts: (i) the total matching distance/time for all successfully matched vertices at $t_{i}$, (ii) the total waiting time for all newly arrived demand vertices, and (iii) the total additional waiting time for all leftover demand vertices, summed across all zones:
\begin{equation}
\begin{aligned}
\label{eq:cost_t_i}
\sum_{z\in Z} \left\{ %\right.
m_z(t_{i}) \cdot p[\chi_z(t_i)] %&
\cdot d[\chi_z(t_i)]
+
\lambda_z(t_i) \frac{\tau^2(t_i)}{2} +
%\\
%&+
m_z(t_{i}) \cdot [1-p[\chi_z(t_i)] \cdot \tau(t_i)
%\left. 
\right\}.
\end{aligned}
\end{equation}

Next, given that the distribution of demand and supply vertices vary slowly over time, we use a continuum approximation scheme to model the overall system dynamics and costs evolution. Instead of tracking the optimal matching decisions (timing and maximum radii) at discrete decision epochs, we look for functions of the optimal pooling intervals and matching radii as trajectories over continuous time: $\tau(t)$ and $r_z(t)$, $\forall z\in Z, t\in[0, \mathcal{T}]$. 
Let vector 
$\mathbf{u}(t)=[\tau(t), r_1(t),\ldots,r_{|Z|}(t)]^T$
%=\begin{pmatrix}
%\tau(t) \\
%\mathbf{r}_{t}
%\end{pmatrix}$ 
represents the control decisions at time $t$, and vector $\mathbf{x}(t)=[m_1(t),\ldots,m_{|Z|}(t), n_1(t),\ldots,n_{|Z|}(t)]^T$
%\begin{pmatrix}
%\mathbf{m}(t) \\
%\mathbf{n}(t) 
%\end{pmatrix}$ 
represents the system state (demand and supply distributions) at time $t$. 
The objective is to minimize the total system-wide cost over the entire planning horizon:
\begin{equation} \label{eq:obj_all}
\phi[\mathbf{x}(\mathcal{T})] + \int_{0}^{\mathcal{T}} \mathscr{L}[\mathbf{x}(t), \mathbf{u}(t)] \text{ d}t.
\end{equation}
Here $\phi[\mathbf{x}(\mathcal{T})]$ represents the penalty incurred by the total number of leftover demand and supply vertices at the end of the horizon $\mathcal{T}$; i.e., 
\begin{equation} \label{eq:penality}
    \phi[\mathbf{x}(\mathcal{T})] = \sum_{z\in Z}\left[m_z(\mathcal{T})+n_z(\mathcal{T})\right]V_z.
\end{equation}
Here $\mathscr{L}[\mathbf{x}(t), \mathbf{u}(t)]$ represents the cost rate incurred at time $t$, which is an approximation of 
Equation \eqref{eq:cost_t_i}:
\begin{equation} \label{eq:obj_t}
\mathscr{L}[\mathbf{x}(t), \mathbf{u}(t)] 
\approx 
\sum_{z\in Z} \left\{ 
\frac{m_z(t) \cdot p[\chi_z(t)] \cdot d[\chi_z(t)]}{\tau(t)}
+
\lambda_z(t) \frac{\tau(t)}{2}
+
m_z(t) \cdot [1-p[\chi_z(t)]]
\right\}.
\end{equation}
In addition, the system dynamics given by Equation \eqref{eq:mn_t_i+1} can be rewritten as a system of differential equations:
\begin{equation} \label{eq:diff_sys}
\dot{m}_z(t) = \lambda_z(t) - \frac{ p[\chi_z(t)]\cdot m_z(t)}{\tau(t)}, 
\quad 
\dot{n}_z(t) = \mu_z(t) - \frac{ p[\chi_z(t)]\cdot m_z(t) }{\tau(t)},
\quad
\forall z\in Z, t \in [0, \mathcal{T}].
\end{equation}

Summarizing the above, the optimal control problem can be formulated as follows:
\begin{align}
    \min \quad 
    & \eqref{eq:obj_all} \nonumber \\
    \text{s.t.} \quad 
    & \eqref{eq:p and d}, \eqref{eq:penality}-\eqref{eq:diff_sys},\nonumber
    \\
    & \frac{1}{\lambda_z(t)V_z} \le \tau(t) \le \mathcal{T}  \text{ and } 0\le r_z(t) \le 1, \quad \forall z\in Z, t\in[0, \mathcal{T}]. \label{cons:tau_r}
\end{align}
Equation \eqref{cons:tau_r} defines the bounds on the control decisions $\tau(t)$ and $r_z(t)$. Specifically, the lower bound on $\tau(t)$ ensures that, on average, at least one demand vertex is present in each zone to trigger the matching decisions.  

\subsection{Solution method}
\label{sec:solution}
Since the system dynamics and cost evolution are nonlinear, directly solving the above formulation to identify the exact time-varying optimal control trajectories is challenging. 
In this section, we adopt an indirect method to solve for a set of  equations provided by the optimality conditions, which is more computationally efficient.

By introducing $\bm{\varphi}(t)$ as the adjoint (costate) vector of Lagrange multipliers associated with the system dynamics, the Hamiltonian of the control problem can be written as follows: 
\begin{equation}
\mathscr{H}[\mathbf{x}(t), \mathbf{u}(t), \bm{\varphi}(t)] = \mathscr{L}[\mathbf{x}(t), \mathbf{u}(t)] + \bm{\varphi}^T(t)\mathbf{f}[\mathbf{x}(t), \mathbf{u}(t)],
\end{equation}
where $\mathbf{f}[\mathbf{x}(t), \mathbf{u}(t)] = \dot{\mathbf{x}}(t)$ represents the system dynamics constraints as defined in Equation \eqref{eq:diff_sys}. Let $\nabla\mathbf{f}(t) = \frac{\partial\mathbf{f}[\mathbf{x}(t), \mathbf{u}(t)]}{\partial\mathbf{x}(t)}$ denote the Jacobian matrix of $\mathbf{f}[\mathbf{x}(t), \mathbf{u}(t)]$ with respect to the state vector $\mathbf{x}(t)$ at time $t$. 
According to the Pontryagin's Minimum Principle (PMP), the optimal solution must satisfy the following conditions to ensure that the Hamiltonian is minimized (or stationary) with respect to infinitesimal variations in the control variables: 
\begin{enumerate}
    \item[(a)] The optimal control vector $\mathbf{u}^*(t)$ minimizes the Hamiltonian at each time step $t$; i.e.,
    \begin{align}\label{eq:u_hamilton}
    \mathbf{u}^*(t) = \underset{\mathbf{u}(t)}{\arg\min} \text{ } \mathscr{H}[\mathbf{x}(t), \mathbf{u}(t), \bm{\varphi}(t)], \quad \forall t\in [0, \mathcal{T}].
    \end{align}
    \item[(b)] The costate vector $\bm{\varphi}(t)$ satisfies the following differential equations: 
    \begin{align} \label{eq:phi_dynamics} 
    \dot{\bm{\varphi}}(t) = -\nabla\mathbf{f}^T(t)\bm{\varphi}(t) - \left[\frac{\partial \mathscr{L}[\mathbf{x}(t), \mathbf{u}(t)]}{\partial \mathbf{x}(t)}\right]^T, \quad \forall t\in [0, \mathcal{T}].
    \end{align}
    \item[(c)]
    The terminal condition at time $t=\mathcal{T}$ is given by: 
    \begin{align}\label{eq:phi_boundary}
    \bm{\varphi}^T(t) = \frac{\partial \phi[\mathbf{x}(t)]}{\partial \mathbf{x}(t)}\Bigg|_{t=\mathcal{T}}.
    \end{align} 
\end{enumerate}

Several indirect methods can be used to solve for the optimal control trajectory $\mathbf{u}^*(t)$ and the associated costate vector $\bm{\varphi}(t)$ based on the above conditions. Common approaches include the shooting methods \citep{betts_survey_1998, passenberg_theory_2012} and the Forward-Backward Sweep Method (FBSM) \citep{lenhart_optimal_2007, mcasey_convergence_2012}, which differ in whether the initial guess is made for the costate or control variables. Here we adopt FBSM, as it is easy to generate an initial and reasonably good guess for the control variables in our problem setting. The key steps %for implementing the proposed FBSM 
are outlined as follows.
\begin{enumerate}
    \item Set the initial guess for the control trajectories as $\mathbf{u}(t) = \mathbf{1},\ \forall t \in [0, \mathcal{T}]$, where the pooling interval is 1 [tu], and the matching radius is 1 [du] for all zones.  
    \item Given the initial system state $\mathbf{x}(0)$ and the current control trajectories $\mathbf{u}(t)$, solve the state trajectories $\mathbf{x}(t)$ forward in time over the interval $[0, \mathcal{T}]$ based on the differential equations given by Equation \eqref{eq:diff_sys}. 
    \item Using the terminal condition at $t = \mathcal{T}$ as specified by Equation \eqref{eq:phi_boundary}, along with the current control trajectories $\mathbf{u}(t)$ and state trajectories $\mathbf{x}(t)$, solve the costate variables $\bm{\varphi}(t)$ backward in time over $[0, \mathcal{T}]$ based on the differential equations given by Equation \eqref{eq:phi_dynamics}. 
    \item Update the control trajectories $\mathbf{u}(t)$ at each $t \in [0, \mathcal{T}]$ by solving the Hamiltonian minimization problem given in Equation \eqref{eq:u_hamilton}. 
    \item If the control trajectories $\mathbf{u}(t)$ have not converged in the past iteration step, return to Step 2 and repeat the process.
\end{enumerate}

%By solving the system of equations in \eqref{eq:Euler-L}, we can determine feasible trajectories of $\bm{\tau}$ and $\mathbf{r}$ that satisfy the local optimality conditions. To further improve the solution, we may check for additional conditions, such as the Legendre-Clebsch (convexity) condition. This can be verified by examining the second-order derivative of the Hamiltonian. More detail of the solution approach will be presented in the full paper.

\section{Numerical Experiments}
\label{sec:numerical}
In this section, we first present a series of numerical experiments to verify the accuracy of the proposed analytical formulas for static homogeneous and heterogeneous RBMPs. Then, we showcase the applicability of the dynamic ST-RBMP in planning mobility services under different operational settings.

\subsection{Verification of formulas for static and homogeneous RBMPs}
A set of Monte Carlo simulations is conducted to verify the analytical analyses presented in Section \ref{sec:RBMP_static}.
All simulations are performed in an Euclidean space (i.e., $p = 2$), while other key parameters, including the spatial dimensionality $D$, region volume $V$, demand vertex density $m$, supply vertex density $n$, and matching radius $r$, are varied.
For each combination of parameter values, 100 RBMP instances are generated.
For each realized problem instance, %given the realized spatial locations of demand and supply vertices within the specified region, 
the optimal matching is obtained using the commercial solver Gurobi.
The sample means and/or standard deviations of the matching probability and/or the average matching distance per demand vertex, averaged across the 100 instances, are computed and compared with the analytical predictions.

\subsubsection{Scaling property}
We begin with examining the scaling behavior of the expected matching distance, $\mathbb{E}[X]$, with respect to the region volume $V$ in static homogeneous RBMPs.
Both two- and three-dimensional spaces are considered (i.e., $D = 2$ and $D = 3$), as they are of particularly interest to real-world applications. Here, we set the base demand density to be $m = 2$ per unit volume, and test four scenarios with supply-to-demand ratios $\frac{n}{m} \in \{1, 1.5, 2, 3\}$. The region volume $V$ is varied from 1 to 50. % to demonstrate how the expected matching distance scales with the increasing volume.  
Figure \ref{fig:dist_V} compares the sample means from the Monte Carlo simulations (discrete markers) with the analytical predictions of $\mathbb{E}[X]$ from Equation \eqref{eq:E_X_V} (continuous curves). % lines with different styles for each supply-to-demand scenario, while the corresponding sample means across all simulated RBMP instances are shown as distinct markers.

\begin{figure}[h]
    \centering
    \includegraphics[width=\textwidth]{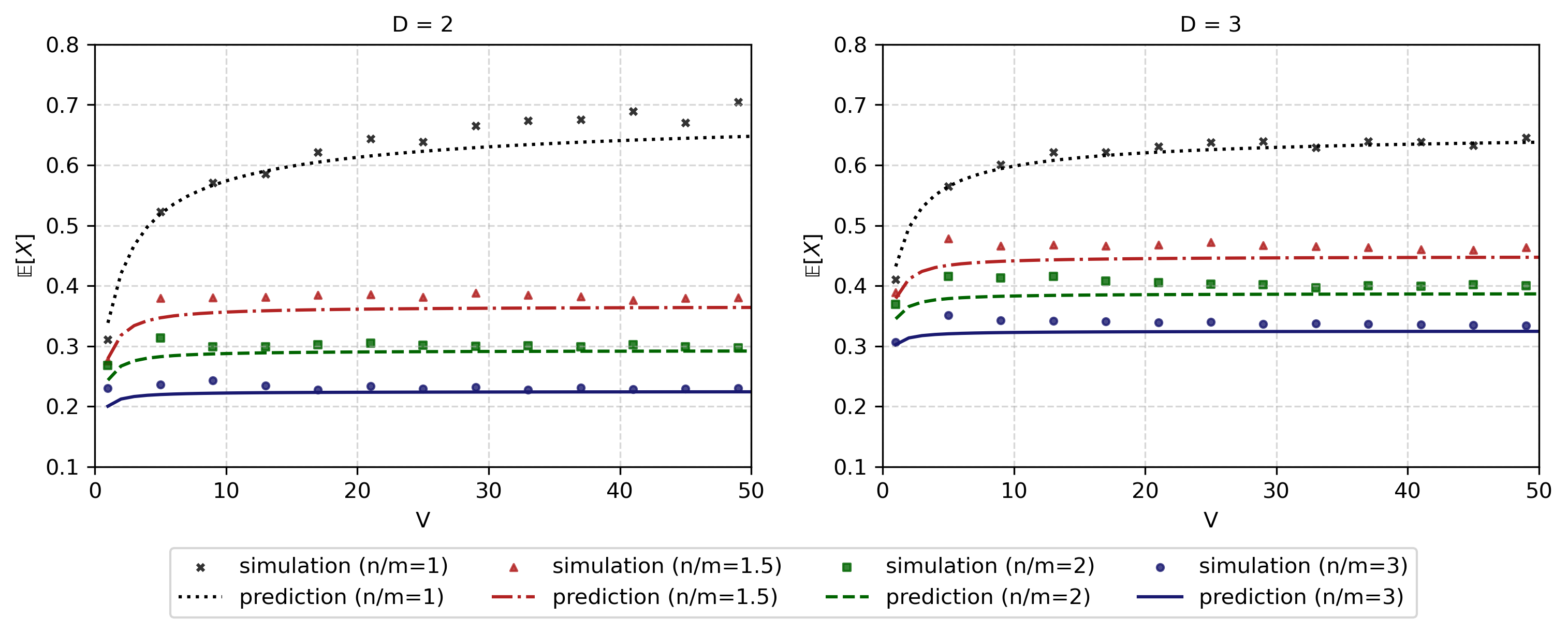}
    \caption{Verification of the Scaling Property.}
    \label{fig:dist_V}
\end{figure}

As shown in Figure \ref{fig:dist_V}, the analytical predictions align closely with the sample means across all parameter combinations. For $D = 2$, the average relative errors in the estimates for $\frac{n}{m} \in \{1, 1.5, 2, 3\}$ are 10.01\%, 7.27\%, 5.82\%, and 5.09\%, respectively. For $D = 3$, the corresponding errors are 4.38\%, 6.10\%, 5.85\%, and 5.20\%.
These results are consistent with the findings in \citet{shen_zhai_ouyang_2024}.\footnote{This reference presented a set of more accurate, but more complex, formulas and detailed comparative analyses for homogeneous RBMPs.} As such, for practical purposes, the proposed formula in Equation \eqref{eq:E_X_V} offers a reasonable trade-off between accuracy and computational efficiency.

In addition, the scaling behavior of the expected matching distance $\mathbb{E}[X]$ is also clearly illustrated in Figure \ref{fig:dist_V}. %,which are consistent with the analytical scaling properties (a)–(c) in Section \ref{sec:scaling}. % across different supply-to-demand scenarios. 
As identified in the analytical properties (a)–(c) in Section \ref{sec:scaling}, $\mathbb{E}[X]$ gradually converges to a finite value as $V$ increases in unbalanced cases, but may go to infinity in the perfectly balanced case. This is consistent with the observations in Figure \ref{fig:dist_V}. When $\frac{n}{m} \in \{1.5, 2, 3\}$,  $\mathbb{E}[X]$ quickly converges to a finite value as $V$ increases. For the balanced case with $\frac{n}{m} = 1$, $\mathbb{E}[X]$ continues to increase with $V$ beyond the tested range of values. 
Notably, greater imbalance (i.e., larger values of $\frac{n}{m}$) leads to faster convergence. 
In addition, the results also indicate that in those unbalanced cases, once $\mathbb{E}[X]$ has converged, the ratio $\frac{n}{m}$ alone (regardless of the exact values of $m$ and $n$) dictates the expected matching distance. This again is consistent with our discussion in Section \ref{sec:scaling}.
%Finally, as noted in property (d), the convergence rate is faster in higher-dimensional spaces. Figure \ref{fig:dist_V} confirms that the convergence is consistently faster in the three-dimensional case ($D = 3$) than in the two-dimensional case ($D = 2$), across all tested scenarios. 
%All these trends are consistent with the analytical scaling properties (a)–(c) in Section \ref{sec:scaling}. 
%In (nearly) balanced RBMPs, the expected matching distance asymptotically increases with the size of the matching region. 
%For more unbalanced cases, the impact of $|V|$ becomes negligible. Additionally, the results also indicate that when demand and supply densities are sufficiently large, the exact values of $m$ and $n$ are no longer necessary, the supply-to-demand ratio $\frac{n}{m}$ alone is sufficient to characterize the expected matching distance.

\subsubsection{Matching radius}
We next verify the analytical formulas for estimating the matching probability $p(\chi)$ and the expected matching distance $d(\chi)$, where $\chi=(m,n,r,V)$, in static homogeneous RBMPs. % under a maximum matching radius.
The simulations are conducted in a unit-volume hyper-ball in two-dimensional space (i.e., $D = 2$, $V = 1$).
Again, we consider four supply-to-demand scenarios with $\frac{n}{m} \in \{1, 1.5, 2, 3\}$ and set $m = 10$. 
For each scenario, the maximum matching radius $r$ is varied from 0 to 1. % to show its effect on both matching probability and distance.

%\begin{comment}
\begin{figure}[ht!]
    \centering
    \begin{subfigure}{\textwidth}
        \centering
        \includegraphics[width=\textwidth]{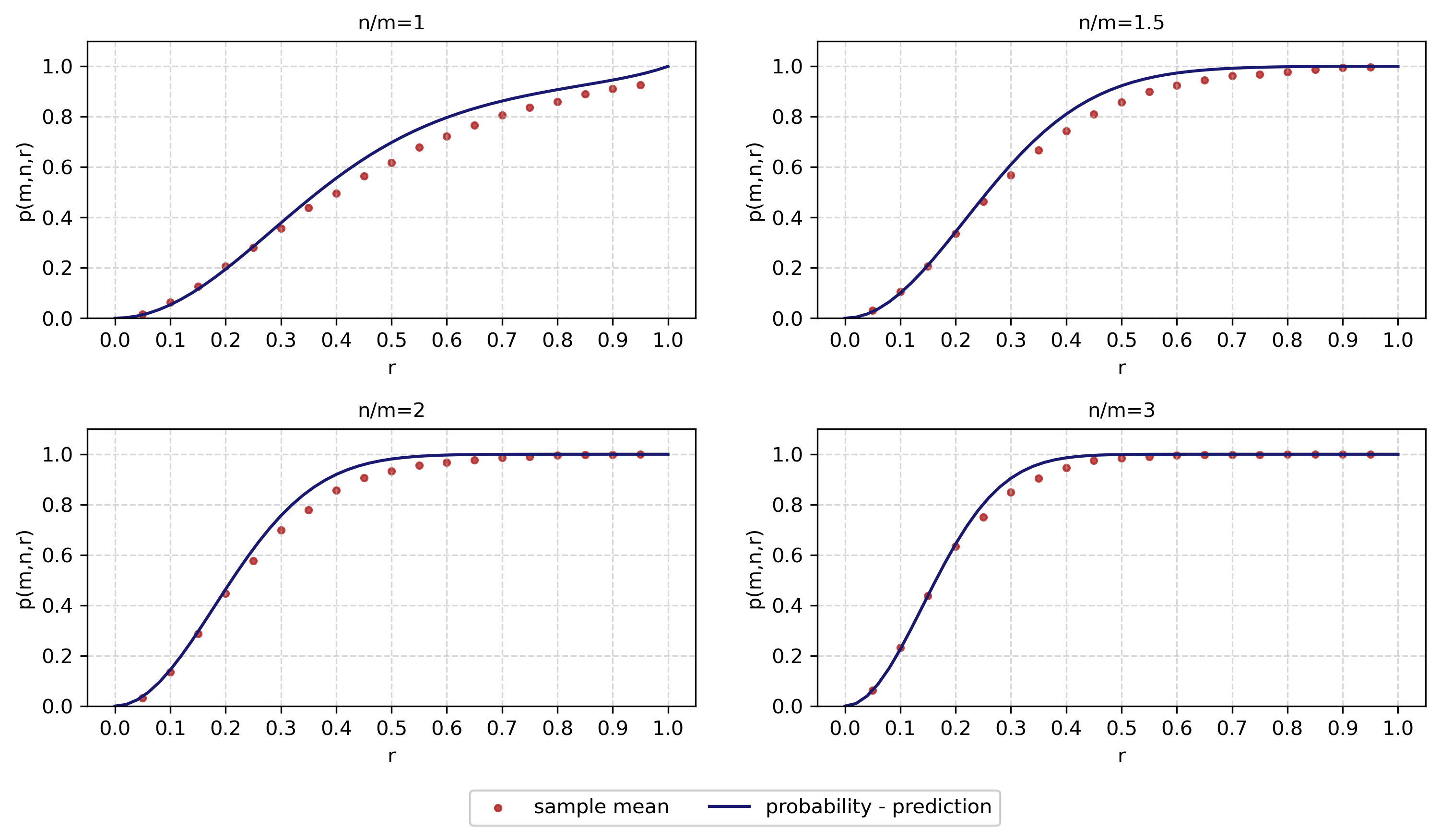}
        \caption{Matching Probability.}
        \label{fig:prob_homo}
    \end{subfigure}
    %\hfill
    %\bigskip
    \begin{subfigure}{\textwidth}
        \centering
        \includegraphics[width=\textwidth]{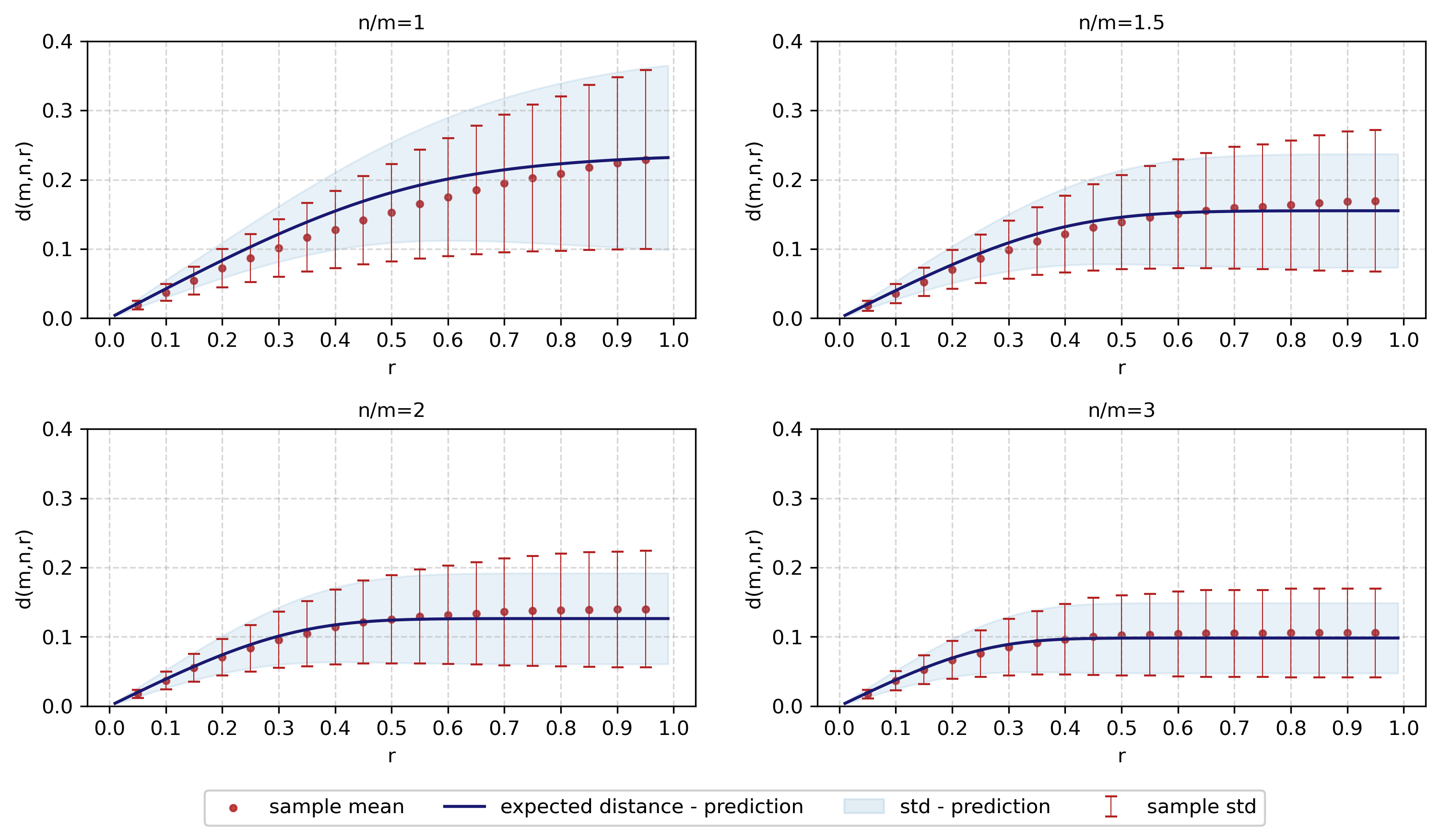}
        \caption{Matching Distance.}
        \label{fig:dist_homo}
    \end{subfigure}
    \caption{Verification of Formulas under Varying Values of Maximum Matching Radius.}
    \label{fig:RBMP_radius}
\end{figure}
%\end{comment}

Figure \ref{fig:RBMP_radius} presents a comparison between the Monte Carlo simulation results and the analytical predictions. The estimates of the matching probability $p(\chi)$ and average matching distance $d(\chi)$, as given by Equations \eqref{eq:p_mnr} and \eqref{eq:d_mnr}, are shown as the solid lines in Figures \ref{fig:prob_homo} and \ref{fig:dist_homo}, respectively. The corresponding sample means obtained from the simulated RBMP instances are represented by the circle markers.
As illustrated, the analytical predictions closely align with the simulation results across all parameter combinations. The average relative errors in the estimates of $p(\chi)$ for $\frac{n}{m} \in \{1, 1.5, 2, 3\}$ are 7.71\%, 4.99\%, 4.16\%, and 1.72\%, respectively. For $d(\chi)$, the corresponding average relative errors are 11.43\%, 6.85\%, 6.10\%, and 5.47\%.

Specifically, we also illustrate the variation in matching distance across realized RBMP instances. The red error bars in Figure \ref{fig:dist_homo} represent the sample standard deviations obtained from the simulations, while the light blue shaded regions correspond to the standard deviations predicted by the analytical formula, computed as the square root of the variance given in Equation \eqref{eq:Var_X}.
These results demonstrate that the proposed probability and distance formulas can provide very accurate estimations. 

In addition, Figure \ref{fig:RBMP_radius} shows how the maximum matching radius $r$ influences both the matching probability $p(\chi)$ and the average matching distance $d(\chi)$ under various supply-to-demand ratios. As shown in Figure \ref{fig:prob_homo}, when $\frac{n}{m} = 1$, the matching probability $p(\chi)$ starts to decline as soon as $r$ drops slightly below 1. In contrast, in more unbalanced scenarios, $p(\chi)$ remains close to 1 and is largely unaffected until $r$ falls below a certain threshold. For instance, when $\frac{n}{m} = 3$, this threshold is approximately 0.5. 
Similar patterns can also be observed for $d(\chi)$. 
These observations can likely be explained by the convergence property identified in \citet{shen_zhai_ouyang_2024}: in highly unbalanced scenarios, competition among demand vertices is minimal, and most demand vertices are matched to their nearest supply vertices. As a result, even when a maximum matching radius is imposed, as long as the maximum distance still exceeds the nearest-neighbor distance, both the matching probability and the expected matching distance remain largely unaffected. In contrast, in more balanced scenarios, this convergence property does not hold. Competition among demand vertices becomes significant, which introduces negative correlation among matches, and some of the demand points will have to be matched to more distant supply vertices. Therefore, when a maximum matching radius is imposed in these scenarios, even a slight reduction in the radius can lead to noticeable decreases in both the matching probability and the expected matching distance.

\subsection{Verification of formulas for static and heterogeneous RBMPs}
We next verify the proposed formulas for static heterogeneous RBMPs. Consider the region of analysis is in a two-dimensional Euclidean space composed of a set $Z$ of $5 \times 5$ unit-volume hexagonal zones (i.e., $V_z = 1, \forall z\in Z$). 
We examine two types of spatial heterogeneity in demand and supply distributions that can be commonly observed in real-world mobility services: uniform and mono-centric.
In the uniform scenario, the number of demand vertices in each zone $m_z, \forall z\in Z$ is randomly generated from a uniform distribution $\text{U}\left((1-\delta)\hat{m}, (1+\delta)\hat{m})\right)$. Here $\hat{m}$ represents a baseline number of demand vertices distributed in all zones, and $\delta \in [0,1]$ controls the degree of heterogeneity; i.e., $\delta =0$ represents a homogeneous case where the number of demand vertices equals $\hat{m}$ in all zones.
In the mono-centric scenario, the number of demand vertices per zone reaches the highest at the center of the region and gradually decreases toward the boundaries: $m_z = (1-\delta)\hat{m} + 2\delta \hat{m}(1-d_z)$.
Here $d_z \in [0,1]$ is the normalized distance from zone $z$'s center to the center of the region, which is the ratio of the absolute distance to the maximum zone-to-zone distance across the region. 

\begin{figure}[h!]
    \centering
    \includegraphics[width=\textwidth]{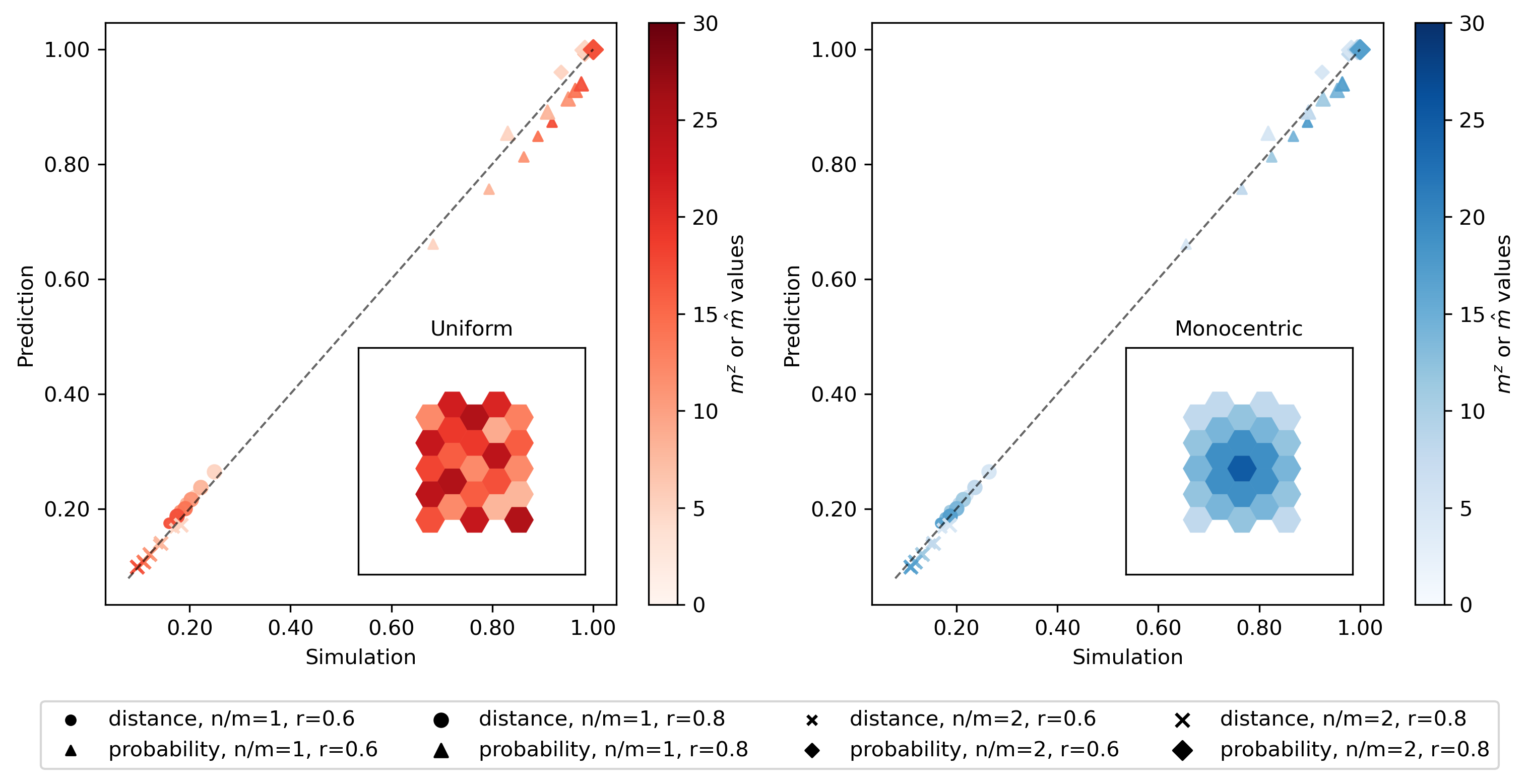}
    \caption{Verification of Formulas under Spatial Heterogeneity.}
    \label{fig:dist and rate}
\end{figure}

For each heterogeneity scenario, we consider two supply-to-demand ratios $n_z/m_z = n/m \in \{1, 2\}$, two levels of maximum matching radius $r_z = r \in \{0.6, 0.8\}$ across all zones, and a range of baseline demand values $\hat{m} \in \{3, 6, 9, 12, 15\}$.
To introduce a sufficient level of heterogeneity in the spatial vertex distribution, we set $\delta = 0.5$ for all cases. 
Then, for each realization of the vertex and radius distributions across all zones, 
we obtain a collective parameter profile $\bm{\chi}$.
Given each profile, 100 heterogeneous RBMP instances are generated. % with random vertex locations. 
Each instance is solved using the commercial solver Gurobi, and the sample means of the matching probability and expected matching distance are recorded.
The corresponding analytical predictions, $\bar{p}(\bm{\chi})$ and $\bar{d}(\bm{\chi})$, are computed using Equation \eqref{eq:P and D}.

Figure \ref{fig:dist and rate} plots the comparison between the simulation results and the analytical predictions. 
For each heterogeneity scenario, a sample realization of the demand distribution profile is shown (similar to the example illustrated in Figure \ref{fig:RBMP_heter_profile}).
Different marker shapes denote different combinations of $n/m$ and $r$ values, while their colors indicate the corresponding values of $\hat{m}$.
As shown, the simulated and predicted results from all cases align closely around the 45-degree line, indicating that Equation \eqref{eq:P and D} provides highly accurate estimations of the simulated results across different heterogeneity patterns.

\subsection{Dynamic matching strategy for ST-RBMP}
Finally, we are ready to showcase the application of ST-RBMP in designing dynamic matching strategies in a mobility service system. To begin, we illustrate how the combination of pooling interval and matching radius affects the matching cost in a dynamic yet homogeneous ST-RBMP at a single time step. 
To this end, we consider a one-time decision on the pooling interval $\tau(0)$ and matching radius $r(0)$ needs to be made at time $t = 0$ for a single zone in $Z = \{1\}$ with unit area size (i.e., $V_1 = 1$). The initial demand density is set as $m_1(0) = 10$, and we evaluate four different supply-to-demand ratios: $\frac{n}{m} \in \{1, 1.5, 2, 3\}$. Demand and supply vertices arrive continuously over time from independent Poisson processes, with respective constant arrival rates $\lambda_1(t) = 10$ and $\mu_1(t) = 20$ for $t \in [0, 5]$.

\begin{figure}[h] 
    \centering        
    \includegraphics[width=1.0\textwidth]{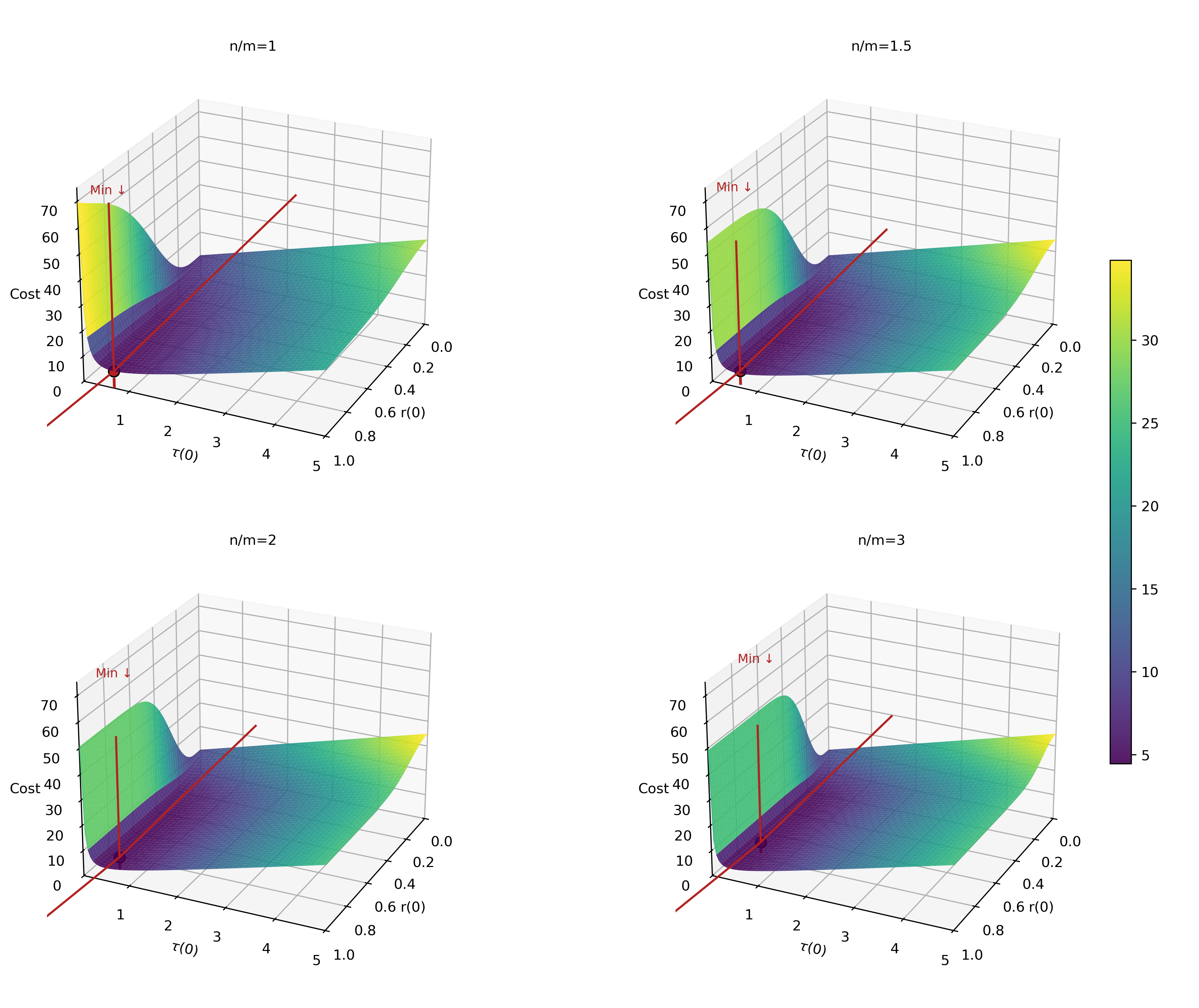}
    \caption{Pooling Interval and Matching Radius at One Single Time Step.}
    \label{fig:dynamic_single_step}
\end{figure}

Figure \ref{fig:dynamic_single_step} plots the surface of the cost function defined in Equation \eqref{eq:cost_t_i} with respect to both the pooling interval $\tau(0) \in [1/\lambda_1(0), 5]$ and the matching radius $r(0) \in [0, 1]$. The optimal combination of $\tau^*(0)$ and $r^*(0)$ that minimizes the cost is marked by the cross arrows on each surface plot.
From the figure, we observe that $\tau^*(0)$ and $r^*(0)$ vary across different supply-to-demand scenarios. 
%, which highlights the importance of tailoring the matching strategy given the location-dependent features. 
Specifically, as $\frac{n}{m}$ increases, both $\tau^*(0)$ and $r^*(0)$ decrease. 
This indicates that, in a system currently with balanced demand and supply, a longer pooling interval and a sufficiently large matching radius could be beneficial. In contrast, in a system with a greater imbalance, a shorter pooling interval and a smaller matching radius becomes preferable. 
%helps ensure that matches are found within an acceptable range. 

These insights help us to better understand the optimal dynamic matching strategies identified in ST-RBMP under spatiotemporal heterogeneity.
Consider four zones with $Z=\{1,2,3,4\}, V_z = 1, \forall z\in Z$, and a planning horizon $[0,5]$ [tu] (i.e., $\mathcal{T}=5$).
At $t=0$, the four zones each have an equal density of demand and supply vertices: $m_1(0)=n_1(0)=3$, $m_2(0)=n_2(0)=5$, $m_3(0)=n_3(0)=10$, $m_4(0)=n_4(0)=20$. 
Over time, we consider three representative demand and supply arrival patterns. 
%, as illustrated in Figure \ref{fig:dynamic_scenario}. 
%For each scenario, the demand rates $\lambda_z(t)$ and supply rates $\mu_z(t)$ are represented by the solid and dashed lines in the two subplots, respectively, with different colors corresponding to different zones.
In the first scenario, the arrival rates of both demand and supply vertices are set as equal and constant over time: $\lambda_z(t) = 2, \mu_z(t) = 2, \forall z \in Z, t\in[0, \mathcal{T}]$. This represents a closed-loop system as discussed in \citet{shen_zhai_ouyang_2024}, in which the service operator manages a fixed fleet of vehicles, and the system reaches equilibrium with equal demand and supply arrival rates. 
In the second scenario, the arrival rates remain constant over time; however, the supply rate exceeds the demand rate: $\lambda_z(t) = 2, \mu_z(t) = 4, \forall z \in Z, t\in[0, \mathcal{T}]$. This represents an open-loop system, where vehicles can enter or exit the system freely, and more new vehicles are expected to enter the system over time. 
The third scenario introduces time-varying demand and supply arrival rates across different zones: $\lambda_z(t) = 2 + z + t, \mu_z(t) = 4 + z + 2t, \forall z \in Z, t\in[0, \mathcal{T}]$, which represents a more time-varying and heterogeneous vertex distribution profile. 

For all three scenarios, we solve the optimal control trajectories $\mathbf{u}^*(t), \forall t \in [0, \mathcal{T}]$, using the dynamic control framework described in Section \ref{sec:solution}. The resulting $\mathbf{u}^*(t)$ and the corresponding evolution of system states $\mathbf{x}(t), \forall t \in [0, \mathcal{T}]$, are illustrated in Figure \ref{fig:pre_result}. In each scenario, the demand and supply densities $m_z(t)$ and $n_z(t)$ in $\mathbf{x}(t)$ are shown as the solid lines with square markers and dashed lines with cross markers, respectively, with different colors indicating different zones. The optimal trajectories of $r_z^*(t)$ in $\mathbf{u}^*(t)$ are plotted as the dotted lines with the corresponding zone colors, and the pooling interval $\tau^*(t)$, shared across all zones, is represented by the black dash-dotted line.

\begin{figure}[h!]
    \centering
    \begin{subfigure}{1\textwidth}
        \centering
        \includegraphics[width=\textwidth]{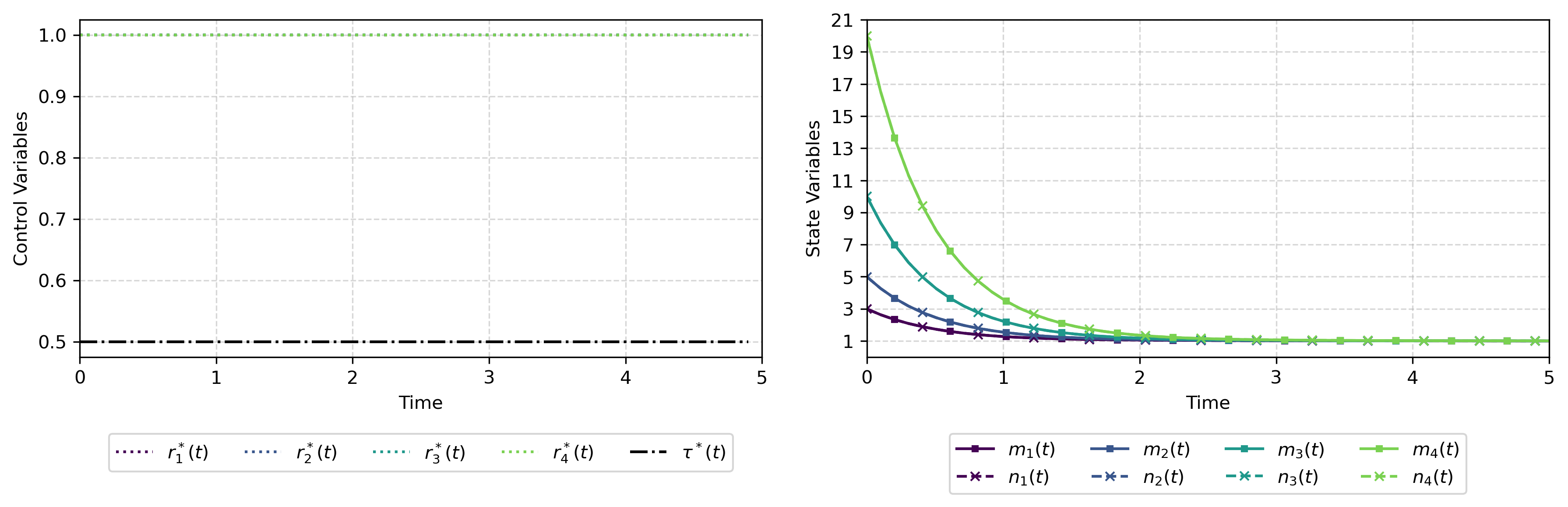}
        \caption{Scenario 1.}
    \end{subfigure}%
    \vfill
    \bigskip
    \begin{subfigure}{1\textwidth}
        \centering
        \includegraphics[width=\textwidth]{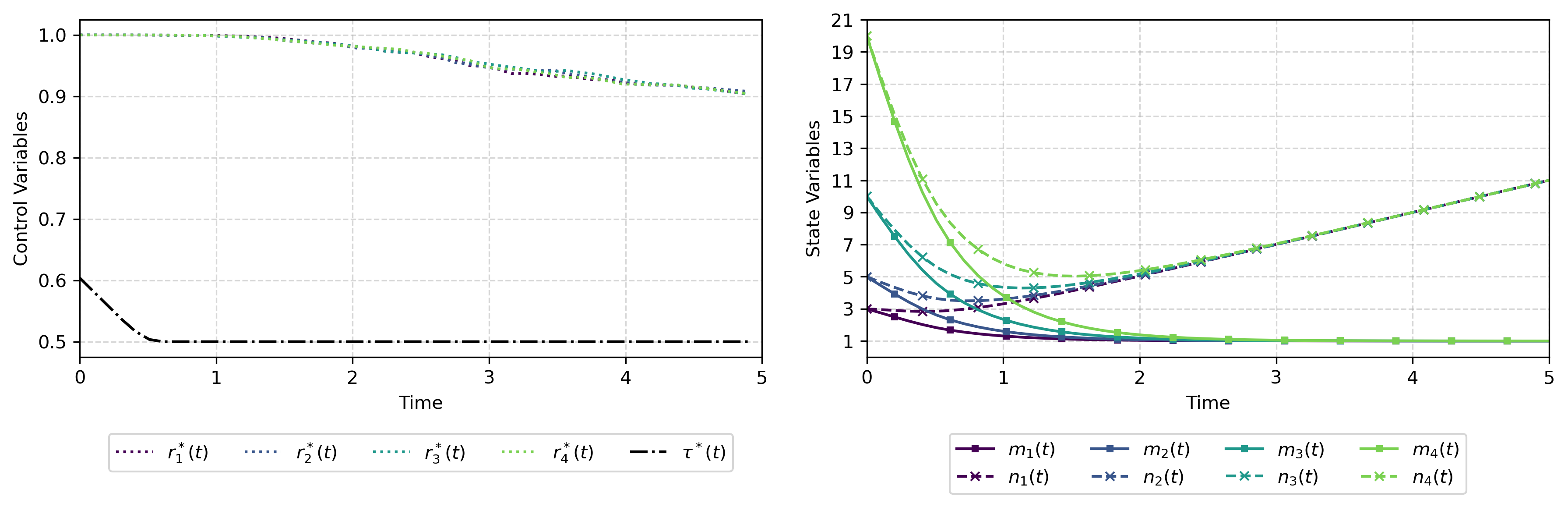}
        \caption{Scenario 2.}
    \end{subfigure}%
    \vfill
    \bigskip
    \begin{subfigure}{1\textwidth}
        \centering
        \includegraphics[width=\textwidth]{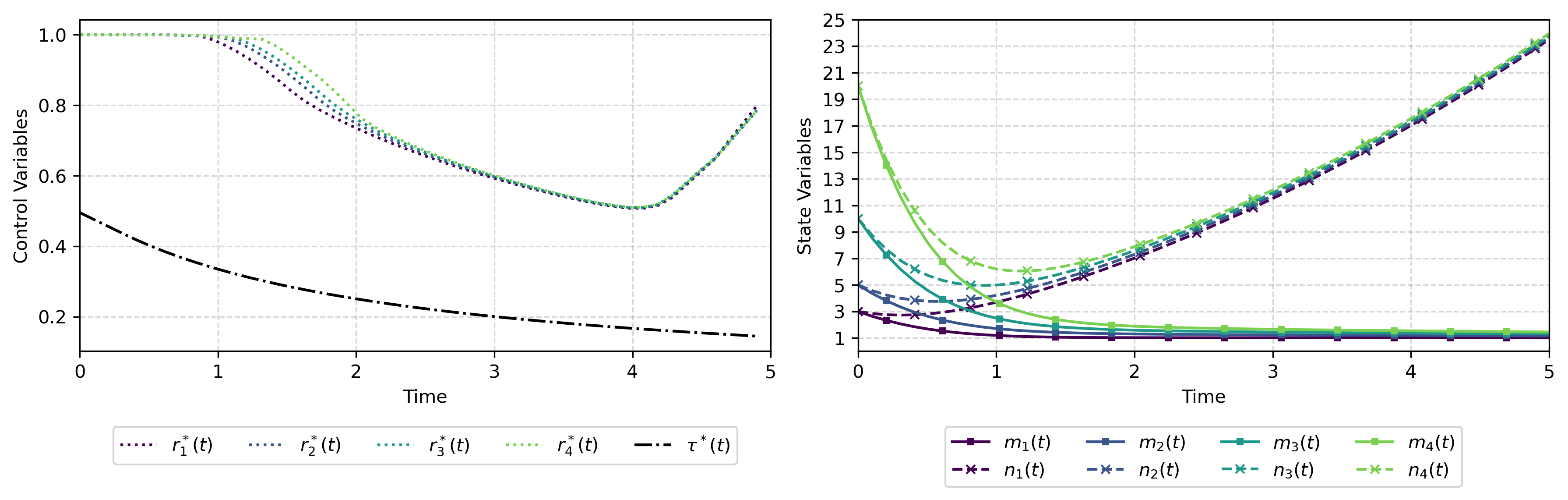}
        \caption{Scenario 3.}
    \end{subfigure}
    \caption{Dynamic Pooling Interval and Matching Radius.}
    \label{fig:pre_result}
\end{figure}

From Figure \ref{fig:pre_result}, we observe $\tau^*(t)$ and $r_z^*(t)$ show quite different patterns over time across different supply-to-demand scenarios. 
Specifically, in the first scenario, 
$\tau^*(t)$ and $r_z^*(t)$ remain constant over the entire planning horizon: $\tau^*(t) = 1/\lambda_z(t) = 0.5, r_z^*(t)=1, \forall z\in Z, t\in [0,5]$. This indicates that instantaneous matching without matching radius is already the optimal strategy, which is consistent with the findings in \citet{shen_zhai_ouyang_2024} for a closed-loop system. 
In the second scenario, we observe that, at the beginning of the planning horizon, $\tau^*(t)$ exceeds $1/\lambda_z(t) = 0.5$, indicating that a longer pooling interval is beneficial. Then, as system evolves, $\tau^*(t)$ gradually decreases to 0.5, suggesting that instantaneous matching becomes optimal.
In addition, $r_z^*(t)$ across all zones follow a similar trend: they remain at 1 at the beginning of the planning horizon and then gradually decreases as the system evolves. This is consistent with the findings in Figure \ref{fig:dynamic_single_step}, when the system becomes more unbalanced, it is beneficial to use shorter pooling interval and smaller matching radius.
In the third scenario, $\tau^*(t)$ constantly decreases over time. Meanwhile, $r_z^*(t)$ across all zones show more fluctuation: they start at $1$ at the beginning, decrease during the middle, and begin to increase toward the end of the planning horizon. 
The reduction of $r_z^*(t)$ in the middle is likely due to the same reason as in the second scenario, since the system becomes more unbalanced during that period of time. 
The increase of $r_z^*(t)$ at the end may be attributed to the penalty cost defined in Equation \eqref{eq:penality}, which penalizes leftover vertices at the end of the horizon. Since the third scenario accumulates most numbers of demand and supply vertices near the end of the planning horizon, it is beneficial to increase the matching probability then to most efficiently clear up the leftover vertices. 
%This also highlights the importance of tailoring the matching strategy to different competing service objectives. 

In addition, in the latter two (unbalanced) scenarios, we observe a consistent decrease in unmatched demand density over the planning horizon, %and this is most pronounced for the latter two scenarios. Meanwhile, 
while the unmatched supply densities %show different rates of change, the separation between them and the demand densities 
keeps increasing. % toward the end of the the planning horizon. 
For the first (balanced) scenario, since the rates of change for both demand and supply are equal, supply and demand densities largely follow the same decreasing trend.
This suggests that the optimal dynamic matching strategy seeks to reduce the demand densities while maximizing the supply-to-demand ratio across all zones over time. This observation is also consistent with the findings in static RBMPs, where a higher supply-to-demand ratio leads to shorter matching distances and higher matching probabilities, ultimately resulting in a lower overall system cost in ST-RBMP.

Finally, in a real-world system, stochasticity in demand and supply arrivals can cause significant variations in the realized problem instances and their solutions. For example, as shown in Figure \ref{fig:dist_homo}, the variation in matching distances is non-negligible.
This type of variation could also cause the random start of WGC, when the system jumps to a different equilibrium state at a certain time. 
To address this, the proposed dynamic matching strategies can be embedded into a closed-loop control framework (i.e., with a sequence of planning horizons). At the start of each planning horizon, the current system state and randomness in demand and supply realizations can be re-evaluated into new system state estimates, and the matching strategy in the next planning horizon can be re-optimized accordingly.

\section{Conclusion}
\label{sec:conclusion}

This paper proposes a new modeling framework to address ST-RBMP under spatiotemporal heterogeneity. 
It begins by analyzing static RBMPs under maximum matching radii and/or spatially heterogeneous vertex distributions.
New closed-form formulas are proposed for estimating the expected matching probability and distance for heterogeneous RBMPs. 
They are derived based on a desirable scaling property identified in homogeneous RBMPs: when the numbers of demand and supply vertices are not (nearly) balanced, the expected matching distance becomes largely independent of the size of the matching region but rather depends primarily on local vertex densities. 
These properties are verified by a series of Monte Carlo simulations, % are conducted to verify the identified scaling property and demonstrate that
and the proposed formulas are shown to provide highly accurate estimates across a wide range of problem settings, including varying values of supply-to-demand ratios, matching radii, and spatial heterogeneity patterns.

The analytical formulas developed for static RBMPs are integrated into the formulation of ST-RBMP in the context of shared mobility services, which dynamically determines the key hyper-parameters for real-time matching, including optimal pooling intervals and maximum matching radii. 
It is formulated as an optimal control problem within a continuum approximation scheme, with the objective to minimize the system-wide costs experienced by both vehicles and customers.
The values of the control variables over time and space are solved from the optimality conditions. 
A series of numerical experiments is conducted to verify the effectiveness of the proposed modeling framework under various demand and supply patterns. 
The results offer theoretical understanding on how dynamic pooling intervals and maximum matching radii impact system costs, and provide valuable managerial insights for mobility service operators in designing their matching strategies. 
In a closed-loop system operating with a fixed fleet of vehicles, instantaneous matching without imposing a matching radius is often optimal. While in an open-loop system, the matching strategies can become more dynamic. 
For example, when the system currently has many leftover passengers and more vehicles are expected to arrive, a longer pooling interval could be beneficial. As the system evolves and the supply-to-demand ratio increases, shorter pooling intervals and smaller matching radii become more favorable. 

This modeling framework effectively captures the spatiotemporal heterogeneity in demand and supply distributions, making it applicable to a wide range of other real-world location-based service systems involving on-demand delivery and resource allocation. Compared to data-driven methods, our analytical model-based approach is more insightful, transferable, and robust across diverse application scenarios.

%% Future Research
The model can be further improved in several directions.
For instance, it currently assumes that demand is always less than or equal to supply ($m \le n$). While this assumption is reasonable for static and homogeneous RBMPs in many application contexts, it may be violated by dynamic and %can be readily adapted by setting the demand density to $\min(m, n)$ and the supply density to $\max(m, n)$. 
heterogeneous RBMPs when supply or demand vertices display strong spatial clustering pattern. As such, future work should relax this assumption, possibly by an extension of the model in Section \ref{sec:RBMP_spatial_heter} which considers a higher level of matching that clears the excessive demand and supply vertices across different zones. 
Moreover, the solution approach based on PMP can be refined to enhance the solution quality. 
For example, additional optimality conditions, such as the Legendre-Clebsch condition, can be examined by evaluating the second-order derivative of the Hamiltonian with respect to the control variables. 
Finally, another interesting direction for future research is to integrate the hyper-parameters determined by the ST-RBMP (i.e., the matching radius and pooling interval) with other tactical-level operational strategies, such as dynamic vehicle routing, empty vehicle repositioning, and vehicle swapping among customers.

%\section*{Acknowledgement}
%Lorem ipsum dolor sit amet, consectetur adipiscing elit.

\newpage
\bibliographystyle{apalike} 
\bibliography{Ref, Dynamic_matching}

\begin{thebibliography}{}

\bibitem[Afeche et~al., 2018]{afeche_ride-hailing_2018}
Afeche, P., Liu, Z., and Maglaras, C. (2018).
\newblock Ride-{Hailing} {Networks} with {Strategic} {Drivers}: {The} {Impact} of {Platform} {Control} {Capabilities} on {Performance}.
\newblock {\em SSRN Electronic Journal}.

\bibitem[Alonso-Mora et~al., 2017]{alonso-mora_-demand_2017}
Alonso-Mora, J., Samaranayake, S., Wallar, A., Frazzoli, E., and Rus, D. (2017).
\newblock On-demand high-capacity ride-sharing via dynamic trip-vehicle assignment.
\newblock {\em Proceedings of the National Academy of Sciences}, 114(3):462--467.
\newblock Publisher: Proceedings of the National Academy of Sciences.

\bibitem[Arnott, 1996]{arnott_taxi_1996}
Arnott, R. (1996).
\newblock Taxi {Travel} {Should} {Be} {Subsidized}.
\newblock {\em Journal of Urban Economics}, 40(3):316--333.

\bibitem[Betts, 1998]{betts_survey_1998}
Betts, J.~T. (1998).
\newblock Survey of {Numerical} {Methods} for {Trajectory} {Optimization}.
\newblock {\em Journal of Guidance, Control, and Dynamics}, 21(2):193--207.
\newblock Publisher: American Institute of Aeronautics and Astronautics.

\bibitem[Caracciolo et~al., 2014]{caracciolo_scaling_2014}
Caracciolo, S., Lucibello, C., Parisi, G., and Sicuro, G. (2014).
\newblock Scaling hypothesis for the {Euclidean} bipartite matching problem.
\newblock {\em Physical Review E}, 90(1):012118.

\bibitem[Castillo et~al., 2017]{castillo_surge_2017}
Castillo, J.~C., Knoepfle, D., and Weyl, G. (2017).
\newblock Surge {Pricing} {Solves} the {Wild} {Goose} {Chase}.
\newblock In {\em Proceedings of the 2017 {ACM} {Conference} on {Economics} and {Computation}}, pages 241--242, Cambridge Massachusetts USA. ACM.

\bibitem[Daganzo, 2010]{daganzo_public_2010}
Daganzo, C.~F. (2010).
\newblock Public {Transportation} {Systems}:{Basic} {Principles} of {System} {Design},{Operations} {Planning} and {Real}-{TimeControl}.

\bibitem[Feng and Niazadeh, 2020]{feng_batching_2020}
Feng, Y. and Niazadeh, R. (2020).
\newblock Batching and {Optimal} {Multi}-stage {Bipartite} {Allocations}.
\newblock {\em SSRN Electronic Journal}.

\bibitem[Hyland and Mahmassani, 2018]{hyland_dynamic_2018}
Hyland, M. and Mahmassani, H.~S. (2018).
\newblock Dynamic autonomous vehicle fleet operations: {Optimization}-based strategies to assign {AVs} to immediate traveler demand requests.
\newblock {\em Transportation Research Part C: Emerging Technologies}, 92:278--297.
\newblock Publisher: Elsevier BV.

\bibitem[Ke et~al., 2021]{ke_system_2021}
Ke, J., Wu, G., Xu, Z., Yang, H., Yin, Y., and Ye, J. (2021).
\newblock System and method for determining passenger-seeking ride-sourcing vehicle navigation.

\bibitem[Lei et~al., 2019]{lei_path-based_2019}
Lei, C., Jiang, Z., and Ouyang, Y. (2019).
\newblock Path-based {Dynamic} {Pricing} for {Vehicle} {Allocation} in {Ridesharing} {Systems} with {Fully} {Compliant} {Drivers}.
\newblock {\em Transportation Research Procedia}, 38:77--97.

\bibitem[Lenhart and Workman, 2007]{lenhart_optimal_2007}
Lenhart, S. and Workman, J.~T. (2007).
\newblock {\em Optimal {Control} {Applied} to {Biological} {Models}}.
\newblock Chapman and Hall/CRC, New York.

\bibitem[Liang et~al., 2023]{liang_enhancing_2023}
Liang, Y., Li, D., Zhao, J., Ding, X., Lian, H., Hao, J., and He, R. (2023).
\newblock Enhancing {Dynamic} {On}-demand {Food} {Order} {Dispatching} via {Future}-informed and {Spatial}-temporal {Extended} {Decisions}.
\newblock In {\em Proceedings of the 32nd {ACM} {International} {Conference} on {Information} and {Knowledge} {Management}}, pages 4702--4708, Birmingham United Kingdom. ACM.

\bibitem[Maciejewski et~al., 2016]{maciejewski_assignment-based_2016}
Maciejewski, M., Bischoff, J., and Nagel, K. (2016).
\newblock An {Assignment}-{Based} {Approach} to {Efficient} {Real}-{Time} {City}-{Scale} {Taxi} {Dispatching}.
\newblock {\em IEEE Intelligent Systems}, 31(1):68--77.
\newblock Publisher: Institute of Electrical and Electronics Engineers (IEEE).

\bibitem[McAsey et~al., 2012]{mcasey_convergence_2012}
McAsey, M., Mou, L., and Han, W. (2012).
\newblock Convergence of the forward-backward sweep method in optimal control.
\newblock {\em Computational Optimization and Applications}, 53(1):207--226.

\bibitem[Mehta, 2013]{mehta_online_2013}
Mehta, A. (2013).
\newblock Online {Matching} and {Ad} {Allocation}.
\newblock {\em Foundations and Trends® in Theoretical Computer Science}, 8(4):265--368.

\bibitem[Najmi et~al., 2017]{najmi_novel_2017}
Najmi, A., Rey, D., and Rashidi, T.~H. (2017).
\newblock Novel dynamic formulations for real-time ride-sharing systems.
\newblock {\em Transportation Research Part E: Logistics and Transportation Review}, 108:122--140.

\bibitem[Ouyang and Yang, 2023]{ouyang_measurement_2023}
Ouyang, Y. and Yang, H. (2023).
\newblock Measurement and mitigation of the “wild goose chase” phenomenon in taxi services.
\newblock {\em Transportation Research Part B: Methodological}, 167:217--234.

\bibitem[Passenberg, 2012]{passenberg_theory_2012}
Passenberg, B. (2012).
\newblock {\em Theory and algorithms for indirect methods in optimal control of hybrid systems}.
\newblock {PhD} {Thesis}, Citeseer.

\bibitem[Psaraftis et~al., 2016]{psaraftis_dynamic_2016}
Psaraftis, H.~N., Wen, M., and Kontovas, C.~A. (2016).
\newblock Dynamic vehicle routing problems: {Three} decades and counting.
\newblock {\em Networks}, 67(1):3--31.

\bibitem[Qin et~al., 2021]{qin_optimizing_2021}
Qin, G., Luo, Q., Yin, Y., Sun, J., and Ye, J. (2021).
\newblock Optimizing matching time intervals for ride-hailing services using reinforcement learning.
\newblock {\em Transportation Research Part C: Emerging Technologies}, 129:103239.

\bibitem[Rao et~al., 2020]{rao_surf_2020}
Rao, G., Choudhury, S., Lingras, P., Savage, D., and Mago, V. (2020).
\newblock {SURF}: identifying and allocating resources during {Out}-of-{Hospital} {Cardiac} {Arrest}.
\newblock {\em BMC Medical Informatics and Decision Making}, 20(S11):313.

\bibitem[Shanks and Jacobson, 2022]{shanks_online_2022}
Shanks, M. and Jacobson, S.~H. (2022).
\newblock Online total bipartite matching problem.
\newblock {\em Optimization Letters}, 16(5):1411--1426.

\bibitem[Shanks et~al., 2023]{shanks_approximation_2023}
Shanks, M., Yu, G., and Jacobson, S.~H. (2023).
\newblock Approximation algorithms for stochastic online matching with reusable resources.
\newblock {\em Mathematical Methods of Operations Research}, 98(1):43--56.

\bibitem[Shen and Ouyang, 2023]{shen_dynamic_2023}
Shen, S. and Ouyang, Y. (2023).
\newblock Dynamic and {Pareto}-improving swapping of vehicles to enhance integrated and modular mobility services.
\newblock {\em Transportation Research Part C: Emerging Technologies}, 157:104366.

\bibitem[Shen et~al., 2021]{shen_path-based_2021}
Shen, S., Ouyang, Y., Ren, S., and Zhao, L. (2021).
\newblock Path-{Based} {Dynamic} {Vehicle} {Dispatch} {Strategy} for {Demand} {Responsive} {Transit} {Systems}.
\newblock {\em Transportation Research Record: Journal of the Transportation Research Board}, 2675(10):948--959.

\bibitem[Shen et~al., 2024]{shen_zhai_ouyang_2024}
Shen, S., Zhai, Y., and Ouyang, Y. (2024).
\newblock Expected bipartite matching distance in a $d$-dimensional $l^p$ space: Approximate closed-form formulas and applications to mobility services.
\newblock \textit{https://arxiv.org/abs/2406.12174}.

\bibitem[Wu et~al., 2022]{wu_graph_2022}
Wu, S., Sun, F., Zhang, W., Xie, X., and Cui, B. (2022).
\newblock Graph {Neural} {Networks} in {Recommender} {Systems}: {A} {Survey}.
\newblock {\em ACM Computing Surveys}, 55(5):97:1--97:37.

\bibitem[Xu et~al., 2020]{xu_supply_2020}
Xu, Z., Yin, Y., and Ye, J. (2020).
\newblock On the supply curve of ride-hailing systems.
\newblock {\em Transportation Research Part B: Methodological}, 132:29--43.

\bibitem[Yang and Gonzales, 2017]{yang_modeling_2017}
Yang, C. and Gonzales, E.~J. (2017).
\newblock Modeling {Taxi} {Demand} and {Supply} in {New} {York} {City} {Using} {Large}-{Scale} {Taxi} {GPS} {Data}.
\newblock In Thakuriah, P.~V., Tilahun, N., and Zellner, M., editors, {\em Seeing {Cities} {Through} {Big} {Data}: {Research}, {Methods} and {Applications} in {Urban} {Informatics}}, pages 405--425. Springer International Publishing, Cham.

\bibitem[Yang et~al., 2020]{yang_optimizing_2020}
Yang, H., Qin, X., Ke, J., and Ye, J. (2020).
\newblock Optimizing matching time interval and matching radius in on-demand ride-sourcing markets.
\newblock {\em Transportation Research Part B: Methodological}, 131:84--105.

\bibitem[Zha et~al., 2016]{zha_economic_2016}
Zha, L., Yin, Y., and Yang, H. (2016).
\newblock Economic analysis of ride-sourcing markets.
\newblock {\em Transportation Research Part C: Emerging Technologies}, 71:249--266.

\bibitem[Zhai et~al., 2024]{zhai_average_2024}
Zhai, Y., Shen, S., and Ouyang, Y. (2024).
\newblock Average {Distance} of {Random} {Bipartite} {Matching} in {Discrete} {Networks}.
\newblock arXiv:2409.18292 [math].

\bibitem[Zhou et~al., 2007]{zhou_bipartite_2007}
Zhou, T., Ren, J., Medo, M., and Zhang, Y. (2007).
\newblock Bipartite network projection and personal recommendation.
\newblock {\em Physical Review E}, 76(4):046115.

\end{thebibliography}

\newpage
\appendix
%\begin{appendices}
\section{Monotonicity of $\left( \frac{1}{x} - 1 \right) \operatorname{Li}_{-\frac{1}{D}}(x)$} 
\label{app:f_i}
\begin{lemma} \label{lemma:f_i}
    For any $s \in [-1, 0)$, function $\left( \frac{1}{x} - 1 \right) \operatorname{Li}_s(x)$ is strictly increasing on $x\in[0,1)$.
\end{lemma}
\begin{proof}
To prove that the smooth function $\left( \frac{1}{x} - 1 \right) \operatorname{Li}_s(x)$ is strictly increasing on the interval $x \in [0, 1)$, it is sufficient to show that its derivative with respect to $x$ is positive. Since $\frac{d}{dx} \operatorname{Li}_s(x) = \frac{1}{x}\operatorname{Li}_{s-1}(x)$, the derivative is 
\begin{align*}
\frac{d}{dx} \left( (\frac{1}{x} - 1) \operatorname{Li}_s(x) \right)
= \left( -\frac{1}{x^2} \right) \operatorname{Li}_s(x) + \left( \frac{1 - x}{x} \right) \cdot \frac{d}{dx} \operatorname{Li}_s(x)  = \frac{1}{x^2} \left[ (1 - x) \operatorname{Li}_{s-1}(x) - \operatorname{Li}_s(x) \right].
\end{align*}
As such, the sign of the derivative is determined by the following function:
\begin{align*}
    g(x,s) = (1 - x) \operatorname{Li}_{s-1}(x) - \operatorname{Li}_s(x) = 
    \sum_{k=1}^\infty x^k \left( \frac{(1-x)k-1}{k^s} \right).
\end{align*}
It is easy to see that, for any $x$, every term inside the summation is monotonically decreasing with respect to $s$ over the interval $s \in [-1, 0]$. As such, the overall summation $g(x,s)$ takes its maximum at $s=-1$ and minimum at $s=0$. It is known that for specific integer values of $s \in \{0, -1, -2\}$, the poly-logarithm function can be simplified: $\operatorname{Li}_{0}(x)=\frac{x}{1-x}$, $\operatorname{Li}_{-1}(x)=\frac{x}{(1-x)^2}$, $\operatorname{Li}_{-2}(x)=\frac{x(1+x)}{(1-x)^3}$.
Hence, it is easy to verify that $g(x,0)=0$ and $g(x,-1)>0$.
Therefore, $g(x, s) > 0$ for all $s \in [-1, 0)$, which implies that the derivative of $\left( \frac{1}{x} - 1 \right) \operatorname{Li}_s(x)$ is strictly positive on $x \in [0, 1)$.
\begin{comment}
Let $a_k = x^k \left( \frac{(1-x)k-1}{k^s} \right)$, the above summation can be partitioned into two parts:
\begin{align*}
\sum_{k=1}^{\lfloor \frac{1}{1 - x} \rfloor} a^k + \sum_{k=\lceil \frac{1}{1 - x} \rceil}^{\infty} a^k,
\end{align*}
where the first part contains only the negative terms and the second part positive terms.

We next show the positive part is always greater than or equal to the absolute value of the negative part for $x\in[0,1)$. 
For the negative part, given $x^k<1$, and $1-(1-x)k<1$, its absolute value is bounded by the following.
\begin{align*}
\Bigg| \sum_{k=1}^{\lfloor \frac{1}{1 - x} \rfloor} a^k \Bigg| < x\sum_{k=1}^{\lfloor \frac{1}{1 - x} \rfloor} \frac{1}{k^s} 
\le x\int_0^{ \frac{1}{1 - x} }\frac{1}{k^s} \text{ d}k = \frac{x}{1-s}\left( \frac{1}{1 - x} \right)^{1-s}.
\end{align*}
\end{comment}
\end{proof}

The following figure numerically plots the monotonic shape of function
$\left( \frac{1}{x} - 1 \right) \operatorname{Li}_{-\frac{1}{D}}(x)$
for $s = -\frac{1}{D}$, where spatial dimension $D \in \{1, 2, 3, 4\}$. As shown, the function value remains close of 1 for small values of $x<0.5$ across all $D$ values. Notably, for $D > 1$, its value stays close to $0$ when $x < 0.8$, and only begins to increase significantly as $x$ approaches $1$. 
\begin{figure}[H] 
    \centering        
    \includegraphics[width=0.85\textwidth]{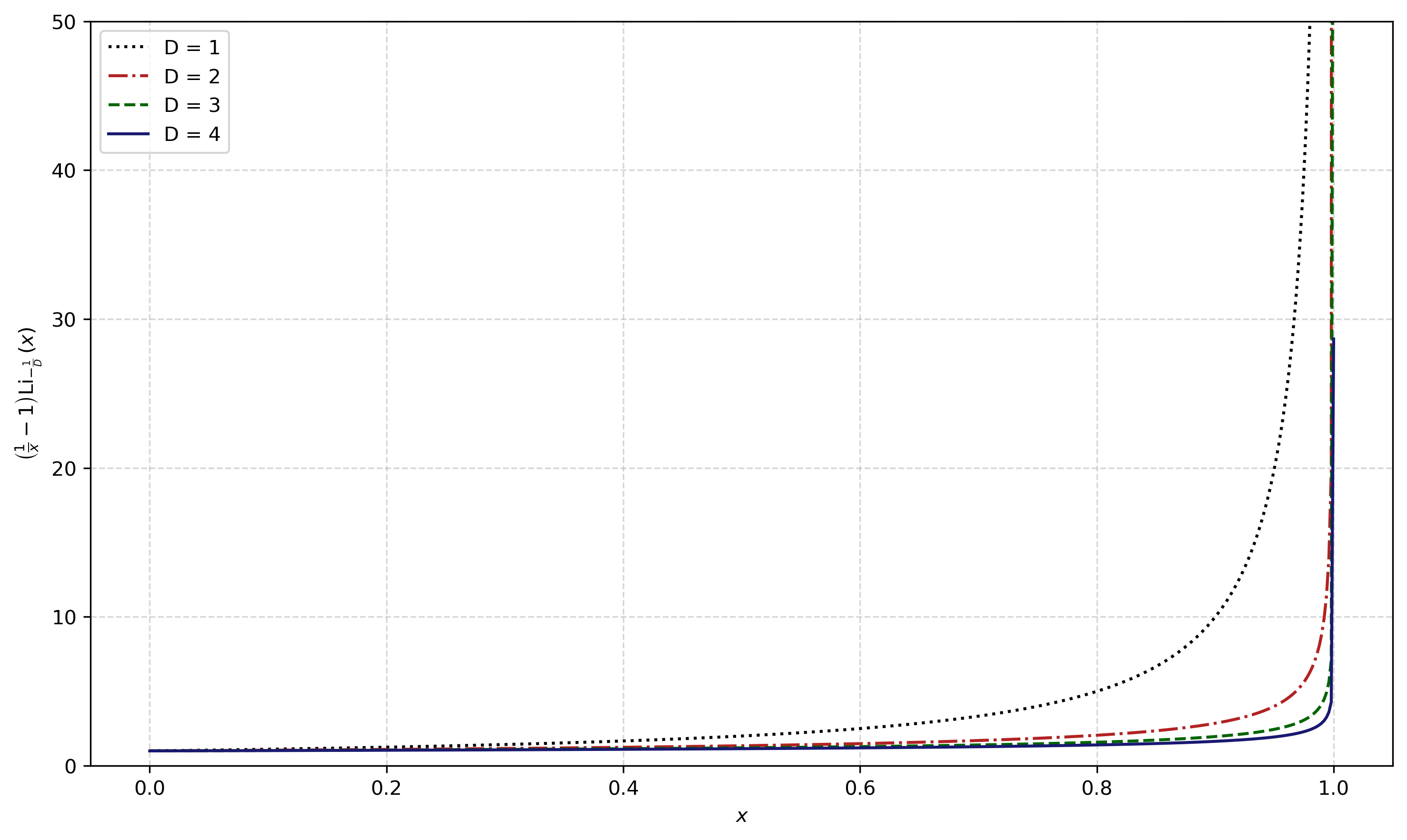}
    \caption{Monotonicity of $\left( \frac{1}{x} - 1 \right) \operatorname{Li}_{-\frac{1}{D}}(x)$.}
    \label{fig:f_x}
\end{figure}

%\end{appendices}

\end{document}